\documentclass[12pt, oneside, a4paper]{amsart}
\usepackage{hyperref}
\usepackage[utf8]{inputenc}
\usepackage{enumerate}
\usepackage[english]{babel}
\usepackage[T1]{fontenc}
\usepackage{ tipa }
\usepackage[utf8]{inputenc}
\usepackage{amsmath, amssymb,amsthm}
\usepackage{array}
\usepackage{graphicx}
\usepackage{float}
\usepackage[normalem]{ulem} 
\usepackage{caption,subcaption}
\usepackage{tikz}
\usepackage{tikz-cd}
\usepackage{ upgreek }
\usetikzlibrary{matrix,arrows,decorations.pathmorphing}
\usepackage{xfrac} 
\usepackage{mathrsfs}
\usepackage{physics}
\usepackage{faktor}

\makeatletter
\newcommand*{\rom}[1]{\expandafter\@slowromancap\romannumeral #1@}
\makeatother

\usepackage{graphicx}
\newtheorem{thm}{Theorem}[section]
\newtheorem{lemma}[thm]{Lemma}
\newtheorem{example}[thm]{Example}
\newtheorem{rmk}[thm]{Remark}
\newtheorem{defi}[thm]{Definition}

\newtheorem{prop}[thm]{Proposition}
\newtheorem{pd}[thm]{Proposition-definition}

\def\A{\mathbb{A}}

\def\C{\mathbb{C}}
\def\D{\mathbb{D}}

\def\H{\mathbb{H}}

\def\N{\mathbb{N}}

\def\P{\mathbb{P}}
\def\Q{\mathbb{Q}}
\def\R{\mathbb{R}}

\def\Z{\mathbb{Z}}

\def\P{\mathbb{P}}

\def\cJ{\mathcal{J}}

\def\sE{\mathscr{E}}

\def\sH{\mathscr{H}}

\DeclareMathOperator{\Rat}{Rat}
\DeclareMathOperator{\rat}{rat}

\DeclareMathOperator{\res}{\mathrm{res}}

\DeclareMathOperator{\PGL}{\mathrm{PGL}}

\DeclareMathOperator{\Per}{\mathrm{Per}}
\newtheorem{introthm}{Theorem}[section]  

\usepackage{mathtools}

\title{Multipliers of a sequence of rational maps}

\author{Chen Gong}

\date{\today}
\thanks{C.G. is  supported by a CSC-202108070159 grant  from the chinese government.}

\begin{document}

\begin{abstract}
We analyze the behavior of multipliers of a degenerating sequence of complex rational maps. We show either most periodic points have uniformly bounded multipliers, or most of them have exploding multipliers at a common scale. We further explore the set of scales induced by the growth of  multipliers. Using Ahbyankar's theorem, we prove that there can be at most $2d-2$ such non-trivial multiplier scales.
\end{abstract} 

\maketitle 

\tableofcontents

\section{Introduction}
Let $f$ be a complex rational map of degree $d \geq 2$, i.e., an endomorphism of the Riemann sphere $\P^{1}(\C)$ written as
\begin{align*}
f \colon~ \P^{1}(\C) &\longrightarrow \P^{1}(\C),\\[3pt]
 [z_0 : z_1] &\longmapsto [P(z_0, z_1) : Q(z_0, z_1)],
\end{align*}
where $P(z_0,z_1)=\sum_{i=0}^{d}a_iz_0^iz_1^{d-i},Q(z)=\sum_{i=0}^{d}b_iz_0^iz_1^{d-i}$ are homogeneous polynomials of degree $d$ having no common factor. The space of all degree-$d$ rational maps is denoted by $\Rat_d(\C)$. Since the representation is unique up to scaling, each $f\in \Rat_d(\C)$ corresponds to a point in $\P^{2d+1}(\C)$ whose homogeneous coordinates are the coefficients of $P$ and $Q$.

Moreover, $\Rat_d(\C)$ coincides with the Zariski open subset of $\P^{2d+1}(\C)$ where the homogeneous resultant does not vanish:
\[
|\Res(f)| = \frac{\Res(P,Q)}{\max\{|a_i|^{2d},|b_j|^{2d}\}} \neq 0.
\]
The resultant function $|\Res|$ is continuous and bounded on $\Rat_d(\C).$

The group of Möbius transformations $\PGL_2(\C)$ acts on $\Rat_d(\C)$ by conjugacy. Since the dynamics of a rational map is invariant under conjugation, it is natural to consider the quotient space
$\rat_d(\C) = \Rat_d(\C) / \PGL_2(\C),$
called the moduli space of degree-$d$ rational maps. Silverman~\cite{SJ98} proved that $\rat_d(\C)$ is an affine variety of dimension $2d-2$. 

It is customary to denote by $[f]$ the conjugacy class of $f$. We say that a sequence $(f_n)_{n\in\N}$ \emph{degenerates in $\rat_d(\C)$} if $([f_n])_{n\in\N}$ escapes every compact subset of $\rat_d(\C)$.

Analogously to the homogeneous resultant, one can define the resultant for a conjugacy class:
\[
|\res([f])| \coloneqq \max_{M \in \PGL_2(\C)} |\Res(M^{-1}\circ f \circ M)|.
\]
The function takes its value in $(0,\max_{f\in\Rat_{d}(\C)}|\Res(f)|]$, as the orbit of $f$ under the action of $\PGL_2(\C)$ is closed, and the resultant $|\Res(f)|$ is a continuous and bounded function.

By definition, the sequence $(f_n)_{n\in\N}$ degenerates in the moduli space if and only if $|\res([f_n])|\to 0.$

\medskip

A point $p\in\P^1(\C)$ is periodic for $f$ if $f^l(p)=p$ for some $l\geq 1$, with the smallest such $l$ called its period. The \emph{multiplier} of $f$ at $p$ is $(f^l)'(p)$. Multipliers are invariant under conjugacy, and their collection provides a natural dynamical fingerprint of $f$. In fact, McMullen~\cite{McM87} proved that, except for the flexible Lattès family, the full collection of multipliers determines the conjugacy class up to finitely many possibilities. Ji–Xie later strengthened the theorem of McMullen in~\cite{JX23} that it suffices to consider the absolute values of multipliers. They further established in~\cite{JX23a} that, on a Zariski open dense subset, the full set of multipliers uniquely determines the conjugacy class. 

\medskip

The asymptotic behavior of multipliers of a degenerating sequence has been studied in several settings.

There are two types of results. The first asserts that $(f_n)_{n\in\N}$ admits an exploding multiplier. For quadratic rational maps, Milnor~\cite{JM93} proved that a sequence degenerates in the moduli space if and only if a multiplier of a fixed point tends to infinity. In the polynomial case, DeMarco–McMullen~\cite{DMM08} (see also~\cite{BH88}) showed that degeneration is again equivalent to the existence of an exploding multiplier. This statement was
recently strengthened by Huguin~\cite{Hug24}, who showed that it suffices to consider
periodic points of period at most two. An analogous statement for cubic rational maps was obtained by Favre in~\cite{Fav25}.

However, degeneration does not always force multipliers to explode. In the flexible Lattès family, all multipliers remain constant. The McMullen family $f_t(z) = z^p + t/z^q$ with $p,q\geq 2$ and $p+q\geq 4$~\cite{McM88,DLU05,QWY12} provides another example: the family degenerates, but all multipliers are uniformly bounded. Luo~\cite{YLuo22} proved the following remarkable statement: a hyperbolic component in $\rat_d(\C)$ contains a degenerating sequence with bounded multipliers if and only if the Julia set of some element in the component is nested.   

Recently, Favre–Rivera-Letelier~\cite{FRL25,Fav25} provides a dichotomy for holomorphic families with a pole at $t=0$: either almost all multipliers diverge at $t=0$, or else all multipliers remain uniformly bounded. 

\medskip

Our first theorem extends this dichotomy to arbitrary degenerating sequences.  

To describe the asymptotic behavior, we employ the ultrafilter technique (see §\ref{sec: ultrafilter}), which provides a consistent way to define limits of sequences in any compact space.  

A non-principal ultrafilter can be viewed as a finitely additive probability measure on $\N$ that is not concentrated on any single point.

Fix such a non-principal ultrafilter $\omega$.  
For each integer $l\ge1$, let $\mathrm{P}_l$ denote the set of all sequences $(p_n)_{n\in\N}$ where each $p_n$ is a periodic point of $f_n$ whose period divides $l$.  
We identify two sequences $(p_n)_{n\in\N}$ and $(p_n')_{n\in\N}$ whenever $p_n=p_n'$ $\omega$-almost surely.  
Since each $f_n$ has exactly $d^l+1$ periodic points of period $l$, it follows that $\mathrm{P}_l$ also has cardinality $d^l+1$.

For notational convenience, define
\[
\lambda(f_n,p_n) := \frac{1}{l} \log^+ \bigl| (f_n^l)'(p_n) \bigr|,
\]
where $\log^+ := \max\{\log, 0\}$.

\begin{introthm}
\label{thm: multiplierdicho}
Fix a non-principal ultrafilter $\omega$. Let $(f_n)_{n\in\N}$ be a sequence of rational maps of degree $d$ degenerating in moduli space. 

Denote by $\chi_{f_n}$ be the Lyapunov exponent of $f_n$, and set \[\chi=\lim_{\omega}-\log|\res([f_n])|\chi_{f_n}.\]
Then exactly one of the following holds:
\begin{enumerate}
    \item  If $\chi>0$, then there for any $\alpha\in(0,\frac{1}{2}\chi)$, we have
    \[
    \lim_{l\to\infty}\frac{1}{d^l}\#\Bigl\{ (p_n)_{n\in\N}\in \mathrm{P}_l : \lambda(f_n,p_n) > -\alpha\log|\res([f_n])|\ \text{ $\omega$-a.s.}\Bigr\} = 1;
    \]
    \item If $\chi=0$, then there exist $C>0$ and an integer $m<d$ such that, for all sufficiently large $l$,  
    \[
    \#\Bigl\{ (p_n)_{n\in\N}\in \mathrm{P}_l : \lambda(f_n,p_n) < C \ \text{ $\omega$-a.s.}\Bigr\} \geq d^l-m^l.
    \]
\end{enumerate}
\end{introthm}
\begin{rmk}
If $(f_n)_{n\in\N}$ converges in the moduli space, then both $(1)$ and $(2)$ hold.
\end{rmk}

\begin{rmk}
In the first case, $\chi$ is the Lyapunov exponent of the non-Archimedean limit map of the sequence~\((f_n)_{n\in\mathbb{N}}\) 
(see \S\ref{sec: constructionNA}).
When \(\alpha > \tfrac{1}{2}\chi\), however, the statement may fail to hold; 
see Example~\ref{eg: fundamental}.

In the second case, one can choose $C=\log d$. In fact, all possible choices of $C$ are completely described in Remark~\ref{rmk: optimalconstant}.
\end{rmk}
\begin{rmk}
The uniformity for meromorphic family  is stronger than the general sequence case, since the equalities of multipliers hold for all parameters \( t \) near $0$ and all periods \( l \in \mathbb{N} \). In contrast, for a general sequence, the \(\omega\)-big set on which these equalities hold may depend on the period \(l \).
\end{rmk}

\begin{example}
Unlike the case of a meromorphic family, case~(2) does not imply that all multipliers remain bounded; indeed, exponentially many periodic points may still have exploding multipliers.  
For example,
\[
f_n(z) = \frac{z^2+n}{1+e^{-n}z^4}.
\] 
See Example~\ref{eg: 1} for details. 
\end{example}

The proof of Favre-Rivera-Letelier in the case of meromorphic family relies on the fact that, for any fixed $l\geq 1$, the function
\[
t \mapsto \max_{f_t^{l}(p_t)=p_t}\log|(f_t^l)'(p_t)|
\]
is subharmonic on $\D^*$ and satisfies \[\max_{f_t^{l}(p_t)=p_t}\log|(f_t^l)'(p_t)| = O(\log|t|^{-1}).\] Hence $\max_{f_t^{l}(p_t)=p_t}\log|(f_t^l)'(p_t)|=o(\log|t|^{-1})$ implies uniform boundedness of multipliers; otherwise, most multipliers explode. This approach does not work in the case of sequences.

We prove the dichotomy for a sequence by using non-Archimedean methods.

Following Favre–Gong~\cite{FG25}, to any sequence of rational maps $(f_n)_{n\in\N}$ one associates a rational map $f_\omega$ defined over a non-Archimedean complete valued field $\sH(\omega)$.  
Favre–Rivera-Letelier~\cite{FRL25} showed that $f_\omega$ either has positive Lyapunov exponent, corresponding to case $(1)$, or has Lyapunov exponent zero.  
In the latter case, most periodic cycles collapse onto a cycle of a monomial map, so that the corresponding multipliers are asymptotically governed, leading to case (2).

\medskip

In contrast to the meromorphic case, one observes another phenomenon: Different periodic points may diverge with significantly different growth behaviors. For instance:

\begin{example}
Consider $f_n(z)=e^{-n}z^3+z^2+nz$, with fixed points
\[
p_{1,n}=0, \quad p_{2,n}=\frac{-1-\sqrt{1-4e^{-n}(n-1)}}{2e^{-n}}.
\]
We have $f_n'(p_{1,n})=n$ while $f_n'(p_{2,n})\asymp e^n$. Thus, the two multiplier sequences explode with distinct growth.
\end{example}

To study the growth of all multipliers, we introduce the notion of \emph{scales}, which provides a refined measurement of their asymptotic behavior.  
A \emph{scale} is represented by a sequence $(\epsilon_n)_{n\in\N}\in (0,1]^\N$, where two sequences $(\epsilon_n)_{n\in\N}$, $(\epsilon'_n)_{n\in\N}$ are identified if and only if $\lim_{\omega}\frac{\epsilon_n}{\epsilon'_n}\in(0,\infty)$.
The scale represented by the constant sequence $(1)$ is called the trivial scale.

For a sequence of periodic points $(p_n)_{n\in\N}\in \mathrm{P}_l$, its \emph{multiplier scale} is defined by
\[
\eta_{p,n} = \frac{1}{1+\lambda(f_n,p_n)}.
\]
The following theorem asserts that, as the period $l$ varies, only finitely many distinct multiplier scales can occur.
\begin{introthm}
\label{thm: intro2d-2}
For any degenerating sequence $(f_n)_{n\in\N}\in\Rat_{d}(\C)^{\N}$, there are at most $2d-2$ non-trivial multiplier scales. In particular, the total number of distinct multiplier scales is at most $2d-1$.
\end{introthm}

\begin{example}
When $d=2$, the bound is optimal. For instance,
\[
f_n(z) = \frac{z(1-z)}{e^{-n}z^2 - n z + n}.
\]
admits three distinct multiplier scales: $(\frac{1}{n})_{n\in\N}$, $(\frac{1}{\log n})_{n\in\N}$, and the trivial scale $(1)_{n\in\N}$. See Example~\ref{eg: 2} for details.
\end{example}

The proof of this theorem  relies on valuation theory. 
Since the moduli space has dimension $2d-2$, we may assume that the non-Archimedean limit $f_\omega$
is defined over a subfield $K\subseteq \sH(\omega)$ of transcendence degree at most $2d - 2$ over $\mathbb{C}$.
The key observation is that each non-trivial multiplier scale defines a valuation ring contained in $K$. 
The expected bound then follows from Abhyankar's inequality, which asserts that any ordered chain of valuation
subrings of $K$ has length at most $\mathrm{deg.tr}(K/\C)$.
 
\subsection*{Acknowledgements}
The author warmly thanks Charles Favre for proposing the problem, for numerous fruitful discussions, and for his careful reading of this manuscript. She is also thankful to Yusheng Luo for valuable exchanges concerning the finiteness of multiplier scales, and to Zhuchao Ji, Junyi Xie, Jit Wu Yap and Yugang Zhang for their insightful remarks. This work forms a part of the author’s Ph.D. thesis.

\section{Dynamics of rational maps over a complete valued field}
\label{sec: preliminary}
\subsection{Berkovich projective line }
Let $k$ be an algebraically closed field endowed with a complete multiplicative norm. We suppose the norm is non-trivial, i.e., there exists $a\in k$ such that $|a|\not=1.$

We introduce the construction of \emph{Berkovich projective line} over $k$, see~\cite{berkovich2012spectral} for more details.

We first define the Berkovich affine line $\A_{k}^{1,an}$ over $k$. It is defined as the set of all multiplicative semi-norms on \( k[T] \) whose restriction to \( k \) agrees with the given norm on $k$. We equipped $\A_{k}^{1,an}$ with the coarsest topology that makes all  evaluation maps of the form \( |\cdot| \mapsto |P| \), for \( P \in k[T] \), continuous. 

Similar to the construction of canonical projective line, the Berkovich projective line $\P_{k}^{1,an}$  is defined by patching together two affine charts $X_i=\A_{k}^{1,an}$ along $\A_{k}^{1,an}\setminus\{0\}$. The construction is as follows: we identify $x_{0} \in X_{0}\setminus\{0\}$  and $x_{1}\in X_{1}\setminus\{0\}$) if and only if $|P(1/T)|_{x_0}=|P(T)|_{x_1}$
for all polynomial $P\in k[T]$.
This construction naturally gives homogeneous coordinates $[z_{0}\colon z_{1}]$ on $\P_{k}^{1,an}$. 
For notationally convience, we usually write 
$z=z_0$, $0 = [0\colon1]$ and $\infty= [1\colon0]$
so that $\P_{k}^{1,an}= \A_{k}^{1,an} \sqcup \{\infty\}$. It follows from \cite[Theorem 1.2.2]{berkovich2012spectral} that the space $\P_{k}^{1,an}$ is compact.

If the field $k$ is Archimedean, by Gelfand-Mazur's theorem and Ostrowski's theorem, there exists $0<\epsilon\leq 1$ such that $k$ is isometric to $\C_\epsilon=(\C,|\cdot|^{\epsilon}),$ where $|\cdot|$ denotes the standard Euclidean norm on $\C$. In particular, if $\epsilon=1$, then $\P_{k}^{1,an}$ is the Riemann sphere $\hat{\C}$.

If the field $k$ is non-Archimedean, that is, if $k$ satisfies the strong triangle inequality $|x+y|\leq \max\{|x|,|y|\} $ for any $x,y\in k$,
then the Berkovich projective line $\P^{1,an}_k$ carries the structure of a compact $\R$-tree: for any \( x, y \in \mathbb{P}_k^{1,an} \), there exists a unique segment \( [x, y] \) connecting them.  Moreover, every point $x \in \P^{1,an}_k$ belongs to one of the following four types:

   \begin{itemize}
       \item[Type-$1$]: $x=\infty$ or there exists $a\in k$ such that for any $P\in k[z]$, $|P|_{x}=|P(a)|$;
       \item[Type-$2$]: there exist $a\in k$ and $r\in |k^*|$ such that 
       \[|P|_{x}=\max_{z\in \bar{B}(a,r)}|P(z)|=\max\{|a_{i}|r^{i}\},\] where $P(z)=a_{d}(z-a)^{d}+...+a_1 (z-a)+a_{0}$, and $\bar{B}(a,r)$ is the closed ball of radius $r$ centered at $a$ in $k$;
      \item[Type-$3$]: there exist $a\in k$ and $r\in \R_{+}\setminus|k^*|$ such that for all $P\in k[z]$, $|P|_{x}=\sup_{z\in \bar{B}(a,r)}|P(z)|=\max\{|a_{i}|r^{i}\}$ as in the previous case;
       \item[Type-$4$]: there exists a sequence of decreasing closed balls $\bar{B}_{n}$  with $\bigcap \bar{B}_{n}=\emptyset$, such that $|P|_{x}=\inf_{n}\max_{z\in \Bar{B}_{n}}|P(z)|$ for all $P\in k[z]$.
   \end{itemize}

Usually, we write $\zeta(a,r)$ for the Type-2 or-3 point associated with the closed ball $\bar{B}(a,r)$, and denote the Gauss point $\zeta(0,1)$ by $x_g$.

The hyperbolic space $\H_{k}$ is defined as the complement in $\P_k^{1,an}$ of Type-1 points.
For any pair of Type-2 or-3 points $x,y\in\H_{k}$ defined by the closed balls $\bar{B}(a_1,r_1)$ and $\bar{B}(a_2,r_2)$ respectively, we 
let \begin{align*}
   d_{k}(x,y)=  2\log\max\{r_{1},r_{2},|a_{1}-a_{2}|\}-\log r_{1}- \log r_{2}.
\end{align*}

    The function $d_{k}(\cdot,\cdot)$ extends continuously to $\H_{k}\times \H_k$ and defines a complete distance
    on $\H_k$. The group of Möbius transformations $\PGL_2(k)$ acts by isometries on $(\H_k,d_{k})$ and it acts transitively on the set of Type-2 points. 
We refer the readers to~\cite[Proposition 2.29]{BR10} for a proof.

\subsection{Julia sets, Lyapunov exponents and periodic cycles}
\label{sec: Julia}

We refer the reader to~\cite{Mi06} for a detailed discussion of the dynamics of complex rational maps. In this section,  we focus on the case where $k$ is a non-trivial, algebraically closed, non-Archimedean field. 

We denote by 
$k^\circ := \{ z \in k : |z| \leq 1 \}$
the valuation ring of $k$, by 
$k^{\circ\circ} := \{ z \in k : |z| < 1 \}$
its maximal ideal, and by $\tilde{k} := k^\circ / k^{\circ\circ}$ its residue field. 

Recall a rational map of degree $d\geq 2$ with coefficients in $k$ is given in homogeneous coordinates by
\[
f : [z_{0}:z_{1}] \longmapsto [P(z_{0},z_{1}) : Q(z_{0},z_{1})],
\]
where $P,Q \in k[z_{0},z_{1}]$ are homogeneous polynomials of degree $d$ without common factors.

It induces a continuous map on the Berkovich projective line $\P^{1,\mathrm{an}}_k$. Concretely, $f=[P:Q]$ sends a rigid point $[z_0:z_1]\in \P^1(k)$ to $[P(z_0,z_1):Q(z_0,z_1)]$, and for a point $x\in \H_k$ with trivial kernel one has
\[
|h(f(x))| = |(h\circ f)(x)|
\quad\text{for all } h\in k(z).
\]

The endomorphism $f\colon \P_{k}^{1,an}\rightarrow \P_{k}^{1,an}$ is continuous, finite, open and surjective, see, e.g.,~\cite[Proposition 4.3]{Jonsson}.
In particular, the pullback of the Dirac measure $\delta_{x}$ supported at any $x\in\P_k^{1,an}$:
\[
f^{*}\delta_x=\sum_{f(y)=x}\deg_{y}f\delta_y\]
 is a positive measure of mass $d$.

Recall from~\cite{Ben19,BR10,FR10} that for any point $x\in\H_k$, the sequence of probability measures $\frac{1}{d^n}(f^n)^*\delta_x$ converges to an ergodic measure $\mu_f$, called the equilibrium measure. This measure does not depend on the choice of $x.$ 

The equilirum measure  $\mu_f$ admits a continuous potential. Hence, for any rational function $h \in k(z)$ we have $\log |h| \in L^1(\mu_f)$, see~\cite[\S 2.4]{FRL06}. 

We set
\begin{equation*}\label{eq:df}
|df|(z) := \frac{|f'(z)|\max\{1,|z|^2\}}{\max\{1,|f(z)|^2\}}
= \frac{|P'Q-Q'P|\max\{1,|z|^2\}}{\max\{|P|^2,|Q|^2\}}.
\end{equation*} The \emph{Lyapunov exponent} of $f$ is defined by
\[
\chi_f := \int \log |df| \, d\mu_f.
\]
This is a finite real number, and it is non-negative by~\cite{Ok15,FRL25}.  

\medskip

The support of $\mu_f$ is the \emph{Julia set} of $f$, denoted by $\cJ(f)$. The Julia set is the minimal closed subset of $\P^{1,\mathrm{an}}_k$ that is backward invariant under $f$. 

\smallskip

Let  $p$ be a periodic point of $f$ with period $l$. If $p$ is of Type-1, it is called \emph{repelling} when its multiplier has modulus $>1$; if $p \in \mathbb{H}_k$, it is called \emph{repelling} when the local degree of $f^l$ at $p$ is at least $2$. By \cite[Theorem 10.88]{BR10}, the Julia set is the closure of the set of repelling periodic points. Moreover, \cite[Theorem 8.7]{Ben19} implies that every repelling periodic point is either of Type-1 or Type-2.

\smallskip

We also recall the following  theorems: the Lyapunov exponent encodes a wealth of the Julia set and the multipliers.

\smallskip

A rational map over $k$ is said to \emph{have potential good reduction} if its Julia set is a single Type-2 point on the Berkovich projective line. It is said to be \emph{Bernoulli} if the convex hull of its Julia set is a non-degenerate geodesic segment in $\H_k$.

We will discuss Bernoulli maps in more details in \S\ref{sec: Bernoullimap}.

\begin{thm}[\cite{Fav25,FRL25}]
\label{thm: dichotomylyapunov}
Suppose the residue field $\tilde{k}$ has characteristic $0$, then the following statements hold: 
\begin{enumerate}
    \item If $\chi_f>0$, then for every $\varepsilon>0$ there exists $C>1$ such that, for all $l\gg 1$, one has
\[
\#\left\{p \in \Per_{l}(f) \,\middle|\, |(f^l)'(p)| > Ce^{\frac{\chi_f}{2}l} \right\} \geq (1 - \varepsilon) d^l;
\] 
\item If $\chi_f=0$, then either $f$ has potential good reduction, or $f$ is Bernoulli.
\end{enumerate}
\end{thm}

\subsection{Tangent map and reduction map}
\label{sec: directiontagentmap}
Let $k$ be a non-trivial algebraically closed non-Archimedean field.  

In this section, we briefly recall the definition of the tangent map of a rational map and explain its relation to the reduction map. For further details, we refer to~\cite{BR10,Ben19}.

\smallskip

Recall that the Berkovich projective line \( \mathbb{P}_k^{1,an} \) has the structure of an \( \mathbb{R} \)-tree: for any \( x, y \in \mathbb{P}_k^{1,an} \), there exists a unique segment \( [x, y] \) connecting them. A \emph{direction} \( \vec{v} \) at a point \( x \in \mathbb{P}_k^{1,an} \) is an equivalence class of points in \( \mathbb{P}_k^{1,an}\setminus\{x\} \), where two points \( y, y' \) are equivalent if the segments \( (x,y] \) and \( (x,y'] \) intersect non-trivially. The set of directions at \( x \) is called the \emph{tangent space} at \( x \), denoted by \( T_x \mathbb{P}_k^{1,an} \). For a given direction \( \vec{v} \), the associated subset of points representing it is denoted by \( U(\vec{v}) \). This is a connected open subset of $\P_k^{1,an}$ whose boundary is exactly the point $x$.

If $x$ is a Type-2 point of $\P^{1,an}_k$, there is a canonical identification
\[
\varphi \colon~ T_x \P^{1,an}_k \xrightarrow{\ \sim\ } \P^1(\tilde{k}).
\]
In particular, if $x=x_g$, then the direction corresponding to the open disc $\{\,|z-a|<1\,\}$ with $|a|\leq 1$ maps to $\tilde a\in\tilde k$, while the direction corresponding to $\{\,|z|>1\,\}$ maps to $\infty$.  

\smallskip

A rational map $f$ of degree at least $1$ induces a tangent map  
\[
T_x f \colon T_{x}\P^{1,an}_k \longrightarrow T_{f(x)}\P^{1,an}_k .
\]
A direction $\vec{v}$ is called \emph{good} if $U(\vec{v})$ does not contain any preimage of $f(x)$. By \cite[Proposition~9.41]{BR10}, if $\vec{v}$ is good, then 
$f(U(\vec{v}))= U(T_x f(\vec{v})),$
while if $\vec{v}$ is bad, then 
$f(U(\vec{v})) = \P^{1,an}_k.$

Choose a normalized representation $f=[P:Q]$ such that the maximum absolute value of the coefficients equals $1$. By reducing the coefficients modulo the valuation ring, one obtains a rational map  
$\tilde{f}=[\tilde{P}:\tilde{Q}]$
over $\tilde{k}$.  This map has degree at most $d.$

If $x_g$ is fixed by $f$, then under the identification $\varphi$, the tangent map coincides with the reduction map $\tilde{f}$. This is summarized in the following commutative diagram (see \cite[Lemma~7.35]{Ben19}):  

\begin{center}
\begin{tikzcd}
T_{x_g}\P^{1,an}_k \arrow[d,"T_{x_g}f"] \arrow[r,"\varphi"] 
  & \P^1(\tilde{k}) \arrow[d,"\tilde{f}"] \\
T_{x_g}\P^{1,an}_k \arrow[r,"\varphi"] 
  & \P^1(\tilde{k})
\end{tikzcd}
\end{center}

Moreover, the following proposition collects some useful properties of the tangent map.
\begin{prop}
\label{prop: basicreduction}
Let $f$ be a rational map over $k$. Then:
\begin{enumerate}
    \item $f$ fixes $x_g$ if and only if $\deg(\tilde{f}) \geq 1$.
    \item If $f$ fixes $x_g$, then $\deg_{x_g}(f) = \deg(\tilde{f})$.
    \item If $f$ fixes $x_g$, then a direction $\vec{v}$ at $x_g$ is bad if and only if $\varphi(\vec{v})$ is a common zero of $\tilde{P}$ and $\tilde{Q}$.
    \item Suppose $f$ fixes $x_g$ and let $\vec{v}$ be a direction at $x_g$. Then there exists $x'$ in the direction $\vec{v}$ such that $f$ maps the segment $[x_g,x']$ linearly and surjectively onto its image, with slope $\deg_{\varphi(\vec{v})}(\tilde{f})$ with respect to the hyperbolic distance.
\end{enumerate}
\end{prop}

\begin{proof}
See~\cite[Theorems~7.22 and~7.34]{Ben19}.
\end{proof}

\subsection{Bernoulli map} 
\label{sec: Bernoullimap}
Let $k$ be a non-trivial, non-Archimedean and algebraically closed  field. 

In this subsection, we  study the localization of periodic points of a Bernoulli map, which serves as an important ingredient in the proof of Theorem~\ref{thm: multiplierdicho}. For additional background and details, we refer the reader to~\cite[\S 5.2]{FR10} and \cite{FRL25}.

Recall a rational map $f:\P^{1,\mathrm{an}}_{k}\to \P^{1,\mathrm{an}}_{k}$ of degree $d\ge 2$
is said to be Bernoulli if the convex hull of its Julia set is a non-degenerate geodesic segment  $I=[a,b]\subseteq \H_k$. It is equivalent to say that $f$ does not have potential good reduction, and there exists a backward invariant segment (i.e., $f^{-1}(I)\subseteq I$) contained in $\H_k$.

In this case, there exist
an integer $m\ge2$ and a partition
\[
a= a_1<b_1\le a_2<\cdots<b_m= b
\]
such that, writing $I_j=[a_j,b_j]\subseteq I$,
\begin{enumerate}
\item $f^{-1}(I)=I_1\cup\cdots\cup I_m$;
\item for each $j=1,\dots,m$, the restriction $f|_{I_j}:I_j\to I$ is  bijective.
\end{enumerate}

When $f$ is Bernoulli,  its restriction $f\colon I_j \to I$ is affine w.r.t the metric $d_\H$, and its slope is equal to $\varepsilon_j d_{j}$  with $d_{j}\in\Z_{\ge 2}$, $\varepsilon_j\in\{-1,+1\}$ and the signs are alternating $\varepsilon_j=-\varepsilon_{j+1}$. In particular, we have $\mathrm{length}(I)=d_j\times \mathrm{length}(I_j)$ so that $\sum_{j=1}^m d_j^{-1} \le 1$. Since $f^{-1}(I) = I_1 \cup \cdots \cup I_m$, we also have
$d=\sum_{j=1}^m d_j$. 

The Julia set $\cJ(f)$ of $f$ satisfies $\cJ(f)=\bigcap_{i=0}^{\infty}f^{-i}(I)$. It is a closed segment  if  $\sum_{j=1}^m d_j^{-1} = 1$, and a Cantor subset of a segment in $\H_k$ if $\sum_{j=1}^m d_j^{-1} < 1$.

\begin{prop}
\label{prop: fixedbernoullilocal}
 Each $I_j$ contains a unique fixed point of $f$, which is repelling and has local degree $d_j$.
\end{prop}
\begin{proof}
Without loss of generality, we suppose $f|_{I_j}$ has positive slope. Then, for any $x\in I_j$,  at least one of the two tangent directions: the one pointing toward $a$ or toward $b$, is backward invariant. By Proposition~\ref{prop: basicreduction}, we have $ \deg_x f=d_j\ge 2$. In particular, any fixed point of $f$ in $I_j$ is repelling.
 
 Since $f$ is affine on $I_j$ with slope $d_j\ge 2$, the equation $f(x)=x$ has at most one solution in $I_j$. Also, since the restriction $f|_{I_j}\colon I_j\to I$ is a bijection and $I_j$ is contained in $I$, a fixed point does exist.
\end{proof}

\begin{prop}\label{prop: fixedpoint}
Let \(p\) be a Type-1 fixed point of \(f\), and let \(x\) be the projection of \(p\) onto \(I\). Then \(f(x) = x\). Moreover, $p$ defines a good direction at $x$.
\end{prop}
\begin{proof}
Denote by $\vec{v}$ the direction at $x$ defined by $p$. Then $U(\vec{v})\cap I=\emptyset.$ Recall if $\vec{v}$ is a bad direction at \( x \), then \( f(U(\vec{v})) = \mathbb{P}_k^{1, \mathrm{an}} \), implying \( U(\vec{v}) \cap f^{-1}(I) \neq \emptyset \), which contradicts the backward invariance of \( I \). Hence, $\vec{v}$ is a good direction at $x$ and \(f(U(\vec{v}))=U(T_{x}\vec{v})\).

Suppose \(f(x) \notin [x, p]\), then \(f(U(\vec{v}))=U(T_{x}\vec{v})\) intersects $I$.  But since \(I\) is backward invariant and \(U(\vec{v})\) does not intersect \(I\), this leads to a contradiction. Thus \(f(x) \in [x, p]\).

We now claim \(f(x) \in I\).  If \(x\) is an endpoint of \(I\), then \(f(x)\) is also an endpoint, hence lies in \(I\).  If \(x\) lies in the interior of \(I\), then \(p \in U^{a,b}\), where \(U^{a,b}\) denotes the connected component of \(\mathbb{P}^{1,\mathrm{an}}_k \setminus \{a,b\}\) containing the open segment \((a,b)\). Since \(p\) is fixed, it lies in \(\bigcap_{i=0}^\infty f^{-i}(U^{a,b})\), implying \(x \in \bigcap_{i=0}^{\infty}f^{-i}(I),\) which is the Julia set $\mathcal{J}(f)$. Thus \(f(x) \in f(\mathcal{J}(f)) = \mathcal{J}(f) \subseteq I\). 

Since $I\cap[x,p]=x$, we conclude \(f(x) = x\).
\end{proof}

\begin{thm}
\label{thm: positionfixedpoints}
Suppose that the characteristic of $\tilde{k}$ is $0$, and that \( x_g \) is a repelling fixed point where the slope of \( f|_I \) at \( x_g \) equals \( \pm\delta \) for some integer $\delta\ge 2$. Then the following statements hold: 
\begin{enumerate}
    \item  If $x_g \in (a, b)$, then $\tilde{f}$ is conjugate to the map $z \mapsto z^{\pm \delta}$, and $f$ admits precisely $\delta \mp 1$ Type-1 fixed points projecting onto  $x_g$. For any such point $p$, there exists a root of unity $\zeta$ with $\widetilde{f'(p)} = \delta\tilde{\zeta}.$
    \item If \( x_g = a \) (resp. \( b \)), then \( \tilde{f} \) is conjugate to a polynomial of degree \( d_1 \) (resp. \( d_k \)). In this case \( f \) has exactly \( d_1\) (resp. $d_k$) Type-1 fixed points projecting to \( x_g \).
\end{enumerate}
\end{thm}
\begin{lemma}
\label{lem: liftfixedpt}
Suppose $x_g$ is a repelling fixed point of $f$. Then any fixed point of $\tilde{f}$ lifts to a  fixed point of $f$.
\end{lemma}
\begin{proof}
Let $f = [P : Q]$ be a homogeneous representation such that the maximal absolute value of the coefficients of  $P$ and $Q$ is equal to $1$.
Since $k$ is algebraically closed, we can factor
\[
z_1 P - z_0 Q = \prod_{i=1}^{d+1} (v_i z_0 - u_i z_1),
\]
with \( \max\{ |u_i|, |v_i| \} = 1 \). The fixed points of $f$ are given by $[u_i\colon v_i].$

Let $\tilde{G}$ be the greatest common divisor of \( \tilde{P} \) and \( \tilde{Q} \).
Write \( \tilde{P} = \tilde{G}\tilde{P}_1 \) and \( \tilde{Q} = \tilde{G}\tilde{Q}_1 \).
A fixed point of \( \tilde{f} \) corresponds to a zero of
$z_1 \tilde{P}_1 - z_0 \tilde{Q}_1,$
which divides
$\prod_{i=1}^{d+1} (\tilde{v}_i z_0 - \tilde{u}_i z_1).
$
Therefore, each fixed point of \( \tilde{f} \) lifts to a fixed point of \( f \).
\end{proof}

\begin{proof}[Proof of Theorem~\ref{thm: positionfixedpoints}]
Assume that $I$ is contained in the geodesic segment joining $0$ and $\infty$. Choose a normalized representation $f=[P:Q]$  so that the maximal absolute value of coefficients of $P$ and $Q$ is equal to $1$. Let $\tilde{G}=\gcd(\tilde{P},\tilde{Q})$ and write $\tilde{P}=\tilde{G}\tilde{P}_1$, $\tilde{Q}=\tilde{G}\tilde{Q}_1$.

Suppose $x_g\in (a,b)$ and the restriction $f|_I$ has positive slope at $x_g$. Since $I$ is backward invariant, every tangent direction at $x_g$ other than those pointing toward $0$ and $\infty$ is good. Hence, for any $\alpha\in\tilde{k}^*$, we have $\tilde{G}(\alpha)\neq 0$.

 By Proposition~\ref{prop: basicreduction}, the reduction $\tilde{f}$ has degree $\delta\ge 2$, and both $0$ and $\infty$ are totally ramified fixed points of local degree $\delta$. Thus $\tilde{f}(z)=\tilde{\lambda}\,z^\delta$ for some unit $\tilde{\lambda}\in\tilde{k}^*$. Choosing $u\in k^{*}$ with reduction $\tilde{u}$ satisfying $\tilde{u}^{\delta-1}=\tilde{\lambda}^{-1}$ and conjugating by $z\mapsto uz$, we may assume $\tilde{f}(z)=z^\delta$. Then the  fixed points of $\tilde{f}$ are exactly the $\delta-1$ roots of $z^{\delta-1}=1$ in $\tilde{k}^*$, which is a reduction of a root of unity $\zeta$. 

 By Lemma~\ref{lem: liftfixedpt}, we obtain $\delta-1$ distinct Type-1 fixed points of $f$ with modulus $1$. Therefore the projections to $I$ all equal $x_g$. Conversely, by Proposition~\ref{prop: fixedpoint}, any Type-1 fixed point $p$ projecting to $x_g$ determines a good direction at $x_g$, so $\tilde{f}(\tilde{p})=\tilde{p}$; thus $p$ coincides with one of the $\delta-1$ lifts constructed above. 
Moreover, since $\tilde{G}(\tilde{p})\not=0$, we have  $\widetilde{f'(p)}=(\tilde{f})'(\tilde{p})=\delta\tilde{\zeta}^{\delta-1}.$

Suppose $x_g = a$.
By Proposition~\ref{prop: basicreduction}, the reduction \( \tilde{f} \) has degree \( d_1 \), and \( \infty \) is a totally ramified fixed point of local degree \( d_1 \).
Hence \( \tilde{f} \) is a polynomial of degree \( d_1 \).
By Lemma~\ref{lem: liftfixedpt}, the fixed points of \( \tilde{f} \) in \( \tilde{k} \) lift to fixed points of \( f \).
We thus obtain \( d_1  \) fixed points whose projection to $I$ is $x_g$.
\end{proof}

\section{Non-standard analysis and higher rank valuations}
\subsection{Ultrafilters}
\label{sec: ultrafilter}
We briefly recall the notion of ultrafilters on the set $\N$ of natural numbers. We refer to \cite{comfort2012theory} for a detailed treatment. 

\begin{defi}
    A subset $\omega$ of the $2^\N$ is called an ultrafilter, if 
    \begin{enumerate}
        \item $\emptyset\notin\omega$;
        \item if $E,F\in\omega$, then $E\cap F\in\omega;$
        \item if $E\in\omega$ and $E\subseteq F$, then $F$ is also in $\omega;$
        \item if $E\subseteq \N$, either $E$ or $E^c=  \N\setminus E$ is in $\omega.$
    \end{enumerate}
    \end{defi}

A set in $\omega$ is called an $\omega$-big set. 
 If a proposition is satified for a $\omega$-big set, we say that this proposition is true $\omega$-almost surely.

An ultrafilter $\omega$ on $\mathbb{N}$ is called principal if there exists some $n \in \mathbb{N}$ such that $\omega = \{\, F \subseteq \mathbb{N} : n \in F \,\}.$ In this case, we denote $\omega$ by $\omega_{n}$.
The existence of non-principal ultrafilters is guaranteed by the Zorn's lemma.

 \smallskip

 One can define the limit of a sequence with respect to an ultrafilter.  The next proposition shows that such a limit always exists when the sequence takes values in a compact set.

\begin{pd}
Fix an ultrafilter $\omega$ on $\N$. Let $K$ be a compact Hausdorff space, and let $(x_n)_{n\in\N}$ be a sequence in $K$. Then there exists a unique point $x \in K$ such that for every neighbourhood $V$ of $x$, one has $x_n \in V$ $\omega$-almost surely. 

This point $x$ is called the $\omega$-limit of $(x_n)_{n\in\N}$, and we write $x=\lim_{\omega} x_n $.
\end{pd}
\begin{proof}
    The uniqueness follows from the fact that $X$ is Hausdorff. To show the existence, one argues by contradiction.  For the sake of contradiction, for every $x\in X$, there exists a neighbourhood $V_{x}$ of $x$ such that 
    \[N_{x}=\{n\in\N, \text{$x_{n}\in V_{x}$}\}\] is not $\omega$-big. The entire space $X$ is covered by a finite number of such $V_{x}$'s. Therefore, $\N$ is covered by a finite number of subsets, each of which is not big. Since the intersection of two big sets is big, it follows that the union of two sets, each of which is not $\omega$-big, is also not $\omega$-big. This leads to a contradiction.
\end{proof}

\begin{rmk}
    Note that if $\omega$ is the principal ultrafilter $\omega_{m}$, then $\lim_{\omega} x_{n} = x_{m}.$
Thus, our main interest lies in limits along non-principal ultrafilters.
\end{rmk}

\subsection{Scales}
\label{sec: scales}
We fix a non-principal ultrafilter $\omega$ and  introduce the notion of a scale to study the asymptotic behavior of  sequences.
\begin{defi}
A scale is an equivalence class of sequences in \((0,1]\), where two sequences \((\epsilon_n)_{n\in\N}\) and \((\eta_n)_{n\in\N}\) are considered equivalent if and only if \((\epsilon_n)_{n\in\N} \asymp_{\omega} (\eta_n)_{n\in\N}\), that is,  there exists a constant \( C > 1 \) such that \( \frac{1}{C} \epsilon_n \leq \eta_n \leq C \epsilon_n \)  \(\omega\)-almost surely. 
\end{defi}
The set of all scales, denoted by $\sE$, is a well-defined set and has the same cardinality \( 2^{\aleph_0}\) as $\R$. This follows from the fact that the set of sequences has cardinality $ 2^{\aleph_0}$, and that the sequences \( \phi_{\alpha} = (\frac{1}{n^{\alpha}})_{n\in\N} \) represent distinct scales for different values of \( \alpha \in (0, \infty) \).

A scale is said to be \emph{trivial} if it is represented by the constant sequence $(1)$. Observe that a sequence \(\epsilon_n\in(0,1]\) represents a trivial scale if and only if \(\lim_{\omega} \epsilon_n > 0\). We denote the trivial scale by $(1).$

 We define an order on scales by writing \(\epsilon \leq \eta\) iff \(\lim_{\omega} \frac{\epsilon_n}{\eta_n} < \infty\). Observe that $\epsilon<\eta$ if and only if $\lim_{\omega}\frac{\epsilon_n}{\eta_n}=0.$ With this ordering, the set of scales forms a totally ordered set, as the $\omega$-limit always exists in $[0,\infty]$. The trivial scale $(1)$ serves as the maximum element.

\begin{rmk}
The set of scales has no minimal element. Moreover, any decreasing sequence of scales admits a lower bound. 
However, any strictly decreasing (resp. strictly increasing) sequence admits no greatest lower bound (resp. least upper bound).

Let $\eta^{k} = (\eta_n^k)_{n\in\N}$, $k\in\N$, be a degenerating sequence of scales, and set $\eta_n = \min_{k=0}^{n}\eta_n^k$. 
Then the scale $\eta = (\eta_n)_{n\in\N}$ is a lower bound of the sequence $(\eta^k)_{k\in\N}$.

Assume now that $(\eta^k)_{k\in\N}$ is strictly decreasing, and let 
$\eta' = (\eta_n')_{n\in\N}$ be any lower bound of $(\eta^k)_{k\in\N}$. 
We construct a non-trivial scale $\epsilon = (\epsilon_n)_{n\in\N}$ such that 
$\frac{\eta'}{\epsilon}$ is also a lower bound of $(\eta^k)_{k\in\N}$. 
This will show that $(\eta^k)_{k\in\N}$ has no greatest lower bound.

By assumption, for each $k\in\N$, we have 
$\lim_{\omega} \frac{\eta_n'}{\eta_n^k} = 0.$
Hence there exists an $\omega$-big set $N_k \subseteq \N$ such that 
$\frac{\eta_n'}{\eta_n^k} < \frac{1}{k}$ for all $n \in N_k$.
Since $(\eta^k)_{k\in\N}$ is decreasing, we may assume that $(N_k)_{k\in\N}$ forms a degenerating sequence of subsets of $\N$. 
Define
\[
\epsilon_n = \frac{1}{\sqrt{k}} \quad \text{for } n \in N_k \setminus N_{k+1}.
\]
Then on $N_k$ we have 
$\frac{\eta_n'}{\epsilon_n \eta_n^k} < \frac{1}{\sqrt{k}},$
so that 
\[
\lim_{\omega} \frac{\eta_n'}{\epsilon_n \eta_n^k} = 0.
\]
Hence the claim.

The same argument applies, mutatis mutandis, to an increasing sequence.
  \end{rmk}

\begin{rmk}
For any two scales $\eta<\eta'$, we write $(\eta,\eta')\coloneqq\{\epsilon\mid\eta<\epsilon<\eta'\}.$ We endow the set of scales with the topology whose open sets are arbitrary unions of intervals of the form $(\eta,\eta')$. Under this topology: 
    \begin{enumerate}
        \item $\sE$ is Hausdorff since for any $\eta<\eta'$, one has $\eta<\sqrt{\eta\eta'}<\eta';$
        \item $\sE$ is not locally compact. 
In fact, for \(\eta < \eta' < 1\), the closed interval \([\eta, \eta'] \coloneqq \{ \epsilon \mid \eta \leq \epsilon \leq \eta' \}\) is not compact. To see this, consider the family of open intervals  
\[
I_\epsilon = \left( \frac{\epsilon}{\log \log \frac{\eta'}{\eta}},\ \min\left\{(1),\epsilon \log \log \frac{\eta'}{\eta} \right\}\right)
\]
for each \(\epsilon \in [\eta, \eta']\). This collection \(\{I_\epsilon\}\) forms an open cover of \([\eta, \eta']\). However, it admits no finite subcover, because each \(I_\epsilon\) contains at most one scale of the form  
$\eta_\alpha = \eta\left( \log \frac{\eta'}{\eta} \right)^\alpha$
for some \(\alpha > 0\).
\item \(\sE\) is not second countable, and hence not metrizable. Indeed, for each \(\alpha > 0\), the intervals  
\[
\left( \left(\frac{1}{n^{\alpha} \log n}\right)_{n\in\N},\ \left(\frac{\log n}{n^{\alpha}}\right)_{n\in\N} \right)
\]
are mutually disjoint. This yields an uncountable family of disjoint open sets, which implies that no countable basis exists.
    \end{enumerate}
\end{rmk}
\subsection{Complex Robinson fields}
\label{sec: complexrobinsonfield}
Fix a non-principal ultrafilter $\omega$.
Given a scale $\epsilon$, one can associate to it a complex Robinson field. This notion was first introduced by Robinson in the real setting in order to formulate non-standard analysis, see~\cite{roe2003lectures}. 

The construction of the field is as follows:
\[
\sH^{\epsilon}(\omega) = \left\{(z_n) \in \mathbb{C}^{\mathbb{N}} \mid \lim_{\omega} |z_n|^{\epsilon_n} < \infty \right\} \Big/ \left\{(z_n) \mid \lim_{\omega} |z_n|^{\epsilon_n} = 0\right\}.
\]
The norm \( |(z_n)| = \lim_{\omega} |z_n|^{\epsilon_n} \) endows \(\sH^{\epsilon}(\omega)\) with the structure of a complete valued field. 

The complex field \(\mathbb{C}\) can be naturally embedded as a subfield of \(\sH^{\epsilon}(\omega)\) via the map  sending each \( c \in \mathbb{C} \) to the equivalence class represented by the constant sequence \( (c) \).

\smallskip

We summarize some properties of $\sH^{\epsilon}(\omega)$ in the following proposition. For a proof, we refer the reader to~\cite[\S 3.2.2]{FG25}.
\begin{prop}
\label{prop: complexrobinsonfield}
The field $\sH^{\epsilon}(\omega)$ is algebraically closed and spherically complete. Moreover, it is non-Archimedean if and only if $\lim_{\omega} \epsilon_n = 0$.  
In particular, if $\epsilon = (1)$, then $\sH^{\epsilon}(\omega)=\mathbb{C}$.  
\end{prop}

There is a canonical map
\[(\mathbb{P}^1(\mathbb{C}))^{\mathbb{N}} \longrightarrow \mathbb{P}^1(\sH^{\epsilon}(\omega))\]
defined as follows: 

Given  $z_n = [z_{0,n} : z_{1,n}] \in \mathbb{P}^1(\mathbb{C})$, where $\max\{ |z_{0,n}|, |z_{1,n}| \} = 1$ for all $n\in\N$, we map the sequence $(z_n)_{n\in\N}$ to
\[z_{\omega}^{\epsilon} = [z_{0,\omega}^{\epsilon} : z_{1,\omega}^{\epsilon}],\]
where each $z_{i,\omega}^{\epsilon} \in \sH^{\epsilon}(\omega)$ is represented by the sequence $(z_{i,n})_{n \in \mathbb{N}}$.

Note that if \(\epsilon\) is the trivial scale, then \(z_\omega^\epsilon\) coincides with the \(\omega\)-limit of \((z_n)_{n\in\N}\) in \(\P^1(\C)\). To emphasize that this limit is taken in the complex projective line, we denote it by \(z_\omega^{\C}\).

\smallskip

Given a non-trivial scale $\epsilon$, one can compare $z_{\omega}^{\epsilon}$ and $z_{\omega}^\C.$

\begin{lemma}
\label{lem: redlimit}
 Let $\epsilon$ be a non-trivial scale, and  $(x_n)_{n\in\N},(y_n)_{n\in\N}\in\C^{\N}.$ If $\widetilde{x}_{\omega}^{\epsilon}=\widetilde{y}_{\omega}^{\epsilon}$, then we have $x_{\omega}^{\C}=y_{\omega}^{\C}$.  In particular, if $|x_\omega^{\epsilon}|<1$, then $x_{\omega}^{\C}=0.$
\end{lemma}

\begin{proof}
We express $x_{n}=[x_{1,n}\colon x_{2,n}]$ and $y_n=[y_{1,n}\colon y_{2,n}]$. We may suppose $\max\{|x_{1,n}|,|x_{2,n}|\}=1$ and $\max\{|y_{1,n}|,|y_{2,n}|\}=1$. By assumption, we have $\lim_{\omega}\epsilon_n\log|x_{1,n}y_{2,n}-x_{2,n}y_{1,n}|<0$. Since $\lim_{\omega}\epsilon_n=0$, we have $\lim_{\omega}\log|x_{1,n}y_{2,n}-x_{2,n}y_{1,n}|=-\infty$, which implies $x_{\omega}^{\C}=y_{\omega}^{\C}.$

If $|x_{\omega}^{\epsilon}|<1$, then $\widetilde{x}_{\omega}^{\epsilon}=0$. Thus, we obtain $x_{\omega}^{\C}=0.$
\end{proof}

 We recall the following result~\cite[Proposition 3.11]{FG25}, which states that a zero of a polynomial over \( \sH^{\epsilon}(\omega) \) can be lifted to zeros of a sequence of complex polynomials.
 If $\epsilon=1$, the lemma is a direct consequence of Rouché's theorem.
\begin{prop}
\label{prop: grouche}
    Let $P$ be a  polynomial with coefficients in $\sH^{\epsilon}(\omega)$ and $\alpha$ be a zero of $P$ with multiplicity $m$. Let $P_n$ be a sequence of complex polynomials of degree $d$  such that $P_{\omega}^{\epsilon}=P$. 
    
    Then there exist sequences $(\alpha_{i,n})_{n\in\N}\in\C^{\N}$, $i=1,\cdots,m$, which satisfy the following conditions: 
    \begin{enumerate}
        \item $\alpha_{i,\omega}^{\epsilon}=\alpha_i$;
        \item $\prod_{i=1}^{m}(z-\alpha_{i,n})$ is a factor of $P_n$ for all $n\in\N.$
    \end{enumerate}
\end{prop}

\subsection{Scales and higher rank valuations} 
\label{sec: higherrank}
We say that a subring \( R \subsetneq K \) is a \emph{valuation ring} of a field \( K \) if for every nonzero element \( x \in K^* \), either \( x \in R \) or \( x^{-1} \in R \). This condition is equivalent to the existence of a function $v\colon R\to\Gamma$ onto a totally ordered abelian group that satisfies the axioms of valuations,  see, e.g.,~\cite{MV00}, and such that  
$R = \{ x \in K^* \mid v(x) \geq 0 \} \cup \{0\}.$

\smallskip

Fix a non-trivial scale \(\epsilon\), i.e., a sequence \((\epsilon_n)_{n\in\N}\) satisfying \(\lim_{\omega} \epsilon_n = 0\). We investigate the valuation rings induced by  scales \(\eta\) with \(\eta \geq \epsilon\).

\begin{lemma}
\label{lem: comparescale}  
Let \(\eta\) be a scale with \(\epsilon < \eta\), and $(z_n)_{n\in\N}\in\C^{\N}$. Then 
\(\lim_{\omega} |z_n|^{\eta_n} < \infty\) implies \(\lim_{\omega} |z_n|^{\epsilon_n} \leq 1\).
\end{lemma}
\begin{proof}
    If \(\lim_{\omega} |z_n|^{\eta_n} < \infty\), then $\lim_{\omega}\epsilon_n\log|z_n|=\lim_{\omega}\frac{\epsilon_n}{\eta_n}\log|z_n|^{\eta_n }\leq 0.$ Thus  \(\lim_{\omega} |z_n|^{\epsilon_n} \leq 1\). 
\end{proof}

The lemma above implies that one can associate to each scale $\eta\geq \epsilon$ the following subsets of \(\sH^{\epsilon}(\omega)\): 
\[
R_{+}^{\eta} = \left\{(z_n)_{n\in\N} \in \mathbb{C}^{\mathbb{N}} \mid \lim_{\omega} |z_n|^{\eta_n} <\infty\right\} \Big/ \left\{(z_n) \mid \lim_{\omega} |z_n|^{\epsilon_n} = 0\right\}.
\]
and 
\[
R_{-}^{\eta} = \left\{(z_n)_{n\in\N} \in \mathbb{C}^{\mathbb{N}} \mid \lim_{\omega} |z_n|^{\eta_n} \leq 1 \right\} \Big/ \left\{(z_n) \mid \lim_{\omega} |z_n|^{\epsilon_n} = 0\right\},
\]

Note that \( (e^{\frac{1}{\eta_n}})_{n\in\N}\in R_{+}^{\eta} \setminus R_{-}^{\eta} \), hence $R_{-}^{\eta}\subsetneq R_{+}^{\eta}\subseteq \sH^{\epsilon}(\omega)$. Observe that by construction, \( R^{\eta} \) contains \( \mathbb{C} \). Also, when $\eta=\epsilon$,  we have $R_{+}^{\eta}=\sH^{\epsilon}(\omega)$ and  $R_{-}^{\eta}=(\sH^{\epsilon}(\omega))^{\circ}$.

\begin{lemma}
For any $\epsilon<\eta<(1)$, $R_{\pm}^{\eta}$ is a valuation ring of $\sH^{\epsilon}(\omega)$.
\end{lemma}
\begin{rmk}
If \( \eta = (1) \), then \( R_{+}^{\eta} \)  is the set of all bounded sequences along $\omega$, which remains a valuation ring, whereas \( R_{-}^{\eta} \) is not a ring.
\end{rmk}
\begin{rmk}
    The maximal ideal of $R_+^{\eta}$ is the set 
    \[\mathfrak{m}_+^\eta=\{(z_n)_{n\in\N}\in\C^{\N}|\lim_{\omega}|z_n|^{\eta_n}=0\}/\{(z_n)_{n\in\N}\in\C^{\N}|\lim_{\omega}|z_n|^{\epsilon_n}=0\};\]
and the maximal ideal of $R_-^{\eta}$ is the set 
 \[\mathfrak{m}_-^\eta=\{(z_n)_{n\in\N}\in\C^{\N}|\lim_{\omega}|z_n|^{\eta_n}<1\}/\{(z_n)_{n\in\N}\in\C^{\N}|\lim_{\omega}|z_n|^{\epsilon_n}=0\}.\]
\end{rmk}
\begin{proof}
    By definition, \( R_{\pm}^{\eta} \) contains the  multiplicative identity $(1)$ and is closed under multiplication. To establish that \( R_{\pm}^{\eta} \) is a ring, it remains to show that it is closed under subtraction.

    Let \( (z_n)_{n\in\N}, (w_n)_{n\in\N} \) represent two elements in \( R_{\pm}^{\eta} \). Since  \[|z_n - w_n| \leq 2\max\{|z_n|, |w_n|\},\] we obtain  
\[
\lim_{\omega} |z_n - w_n|^{\eta_n} \leq \max\{\lim_{\omega} |z_n|^{\eta_n},\lim_{\omega}  |w_n|^{\eta_n} \}<\infty.
\]  
Thus, \( R_{\pm}^{\eta} \) is closed under subtraction, completing the proof that it is a ring.

 Let \( z \in \sH^{\epsilon}(\omega), \) and choose a representative \( (z_n)_{n\in\N} \in \mathbb{C}^{\mathbb{N}} \). If  $\lim_{\omega} |z_n|^{\eta_n} <\infty,$
then \( z \in R_{+}^{\eta} \). On the other hand, if  
$\lim_{\omega} |z_n|^{\eta_n}=\infty,$ then  
$\lim_{\omega} \left| \frac{1}{z_n} \right|^{\eta_n} =0,$ which implies \( z^{-1} \in R_{+}^{\eta} \). Hence, \( R_{+}^{\eta} \) is a valuation ring. Similarly, one can show $R_{-}^{\eta}$ is also a valuation ring.
\end{proof}

\smallskip

We explore how $R_\pm^{\eta}$ varies depending on $\eta$.
We define a partial order on the set \( S = [\epsilon, (1)) \times \{+, -\} \) as follows:  
\[
(\eta, -) < (\eta, +), \quad \text{and} \quad (\eta', +) < (\eta, -) \quad \text{for all } \eta' > \eta.
\]

For any $\sigma=(\eta,*)$ with $*\in\{+,-\}$, we associate a ring $R(\sigma)=R_*^{\eta}.$
\begin{prop}
\label{prop: limitring}
 The map $(\eta,*)\mapsto R_{*}^{\eta}$ is strictly increasing with respect to the above partial order on \(S\). 

Moreover, if $\sigma=(\eta,*)$, then the following set equalities hold:  
\begin{enumerate}
    \item $\bigcup_{\sigma' < \sigma} R(\sigma')= R_{-}^{\eta}$;
    \item $\bigcap_{\sigma' > \sigma} R(\sigma') = R_{+}^{\eta}.$
\end{enumerate}
\end{prop}

\begin{proof}
Lemma~\ref{lem: comparescale} directly implies for any $\eta<\eta'$, one has \( R_{+}^{\eta'} \subseteq R_{-}^{\eta} \). To show that the inclusion is strict, consider  $\left(e^{\frac{1}{\sqrt{\eta_n\eta_n'}}}\right)_{n\in\N} \in R_{-}^{\eta} \setminus R_{+}^{\eta'}.$ Also note that  $R_{-}^{\eta}\subsetneq R_{+}^{\eta}. $ Therefore, the map $(\eta,*)\mapsto R_{*}^{\eta}$ is strictly increasing.

We now prove the set equalities. First, observe the natural inclusions:  
$\bigcup_{\eta < \eta'} R_{-}^{\eta'} \subseteq \bigcup_{\eta < \eta'} R_{+}^{\eta'} \subseteq R_{-}^{\eta}$ and  $R_{+}^{\eta} \subseteq \bigcap_{\eta > \eta'} R_{-}^{\eta'} \subseteq \bigcap_{\eta > \eta'} R_{+}^{\eta'}.$ Thus, it suffices to show  $R_{-}^{\eta} \subseteq \bigcup_{\eta < \eta'} R_{-}^{\eta'}$ and  $ \bigcap_{\eta > \eta'} R_{+}^{\eta'}\subseteq R_{+}^{\eta} .$

For the first inclusion, assume \(\lim_{\omega} |z_n|^{\eta_n} \leq 1\). If \((z_n)_{n\in\N}\) is almost surely bounded, then we have \(z_{\omega}^{\eta} \in  R_{-}^{\sqrt{\eta}}\subseteq \bigcup_{\eta < \eta'} R_{-}^{\eta'}\). If \(\lim_{\omega} |z_n| = \infty\), then $\lim_{\omega}|z_n|^{\eta_n}=1$.  Define  
$\eta_n' = \sqrt{\frac{\eta_n}{\log |z_n|}}.$
Note that \(\eta' > \eta\) and \(\lim_{\omega} |z_n|^{\eta_n'} = 1\), it follows that \(R_{-}^{\eta} \subseteq \bigcup_{\eta < \eta'} R_{-}^{\eta'}\).

For the second inclusion, assume \(\lim_{\omega} |z_n|^{\eta_n} = \infty\). We define  \[\eta'_n = \max \left\{ \sqrt{\frac{\eta_n}{\log |z_n|}}, \epsilon_n \right\}.\] Since \(\epsilon \leq \eta' < \eta\) and \(\lim_{\omega} |z_n|^{\eta_n'} = \infty\), we conclude $(z_n)_{n\in\N}\notin \bigcap_{\eta > \eta'} R_{+}^{\eta'}$. Thus, $  \bigcap_{\eta > \eta'} R_{+}^{\eta'}\subseteq R_{+}^{\eta}.$   
\end{proof}

\subsection{Finiteness of scales on fields of finite transcendence degree}

The following result is a key ingredient in the proof of Theorem~\ref{thm: intro2d-2}.

 \begin{thm}
 \label{thm: numbervaluationring}
Let \( K \) be a field with \( \C \subseteq K \subseteq \sH^{\epsilon}(\omega) \). Suppose that \( K \) has finite transcendence degree \( d \) over $\C$. Then there exist \( l \leq d \) distinct scales  
$\epsilon \leq \eta^1 < \eta^2 < \dots < \eta^l < 1$ such that
\[
R_-^{\eta^i} \cap K \subsetneq R_+^{\eta^i} \cap K \quad \text{for each } i = 1, \dots, l.
\]
Moreover, for any \( \eta \) satisfying \( \eta^i < \eta < \eta^{i+1} \), one has  
\[
R_+^{\eta} \cap K = R_-^{\eta} \cap K = R_+^{\eta^{i+1}} \cap K = R_-^{\eta^i} \cap K.
\]
\end{thm}
 \begin{proof}
Observe that for each $\eta$, the ring $R_*^{\eta}\cap K$ is either equal to $K$, or a proper valuation subring of $K$.  By Proposition~\ref{prop: limitring}, the map $(\eta,*)\mapsto R^{\eta}_*\cap K$ defines a chain of valuation subrings of \(K\). Moreover, all these valuation rings contain \(\mathbb{C}\), and all corresponding valuations are trivial on \(\mathbb{C}\).

 The number of distinct valuation rings in this chain is therefore bounded above by the rank of the finest valuation ring appearing in the image of the map. Since the rank of a valuation ring is bounded by the transcendence degree of the field extension \(K/\mathbb{C}\), we conclude that the image of the map \(\sigma \mapsto R_*^{\eta} \cap K\) with $R_*^{\eta}\cap K\not=K$  has cardinality at most $d$. (See \cite[\S5]{MV00}). 

Therefore, there exist scales $\epsilon\leq\eta^1 < \cdots < \eta^{l}<(1)$ such that the jump $R_-^\eta\cap K\not=R_+^{\eta}\cap K$  occurs if and only if $\eta=\eta^i$ for some $1\leq i\leq l.$ 
Then $R(\eta^i, -)$ defines distinct valuation rings, so we must have $l \leq d$. By construction, for all $\eta$ with $\eta^i < \eta < \eta^{i+1}$, the ring $R_{-}^{\eta} \cap K = R_+^{\eta} \cap K$ is constant. Therefore, this common ring is given by  
\[
\bigcup_{\eta > \eta^i} R_{-}^{\eta} = \bigcap_{\eta < \eta^{i+1}} R_{-}^{\eta}.
\]  
Applying Proposition~\ref{prop: limitring}, we conclude that  
$R_{-}^{\eta^i} = R_{+}^{\eta^{i+1}},$ 
which completes the proof.
\end{proof}

\section{Proof of the main results}

\subsection{Construction of the non-Archimedean limit map}
\label{sec: constructionNA}
Fix a sequence of rational maps $(f_n)_{n\in\N}\in\Rat_{d}(\C)^{\N}$, and a non-principal ultrafilter $\omega$. We can define a limit map of $(f_n)_{n\in\N}$ along $\omega$, which is a rational map over a non-Archimedean complete valued field. For further details, we refer the reader to~\cite{FG25}.

\smallskip

Write $f_n=[P_n\colon Q_n]$ with $P_n(z_0,z_1)=\sum_{i=0}^d a_{i,n}z_0^i z_1^{d-i}$ and $Q_n(z_0,z_1)=\sum_{i=0}^d b_{i,n}z_0^i z_1^{d-i}$, where $\max_{i,j}\max\{|a_{i,n}|,|b_{j,n}|\}=1.$

Recall that the homogeneous resultant of $f_n$ is defined as $|\Res(f_n)|=|\Res(P_n,Q_n)|.$ 
Let \[
\epsilon(f_n)= \left(-\log\frac{|\Res(f_n)|}{C_d}\right)^{-1},\] where
$C_d = e\sup_{\Rat_d(\C)}|\Res(f)|.
$
 By definition, we have $0<\epsilon(f_n)<1.$

Note that the sequences $(a_{i,n})_{n\in\N}$ and $(b_{j,n})_{n\in\N}$ define elements $a_{i,\omega}^{\epsilon(f)}$ and $b_{j,\omega}^{\epsilon(f)}$ in the complex Robinson field $\sH^{\epsilon(f)}(\omega)$ introduced in \S\ref{sec: complexrobinsonfield}. This yields two polynomials $P_\omega^{\epsilon(f)}$ and $Q_\omega^{\epsilon(f)}$ with coefficients in $\sH^{\epsilon(f)}(\omega)$.

Observe that 
\begin{align*}
\label{eqq: resultant} |\Res(P_\omega^{\epsilon(f)},Q_\omega^{\epsilon(f)})|
=\lim_{\omega}|\Res(P_n,Q_n)|^{\epsilon(f_n)}
= e^{-1}>0.
\end{align*}

In particular, $P_\omega^{\epsilon(f)}$ and $Q_\omega^{\epsilon(f)}$ have no common factor, so we obtain a rational map
$$
f_\omega^{\epsilon(f)}=[P_\omega^{\epsilon(f)} : Q_\omega^{\epsilon(f)}]
$$
of degree $d$ over $\sH^{\epsilon(f)}(\omega)$.

\smallskip

Favre and Gong~\cite[Corollary 6.9]{FG25} established the following asymptotic behavior of Lyapunov exponents.

\begin{thm}
\label{thm: asymtlyapunov}
  The Lyapunov exponents satisfy the relation 
\[
\chi_{f_{\omega}^{\epsilon(f)}} 
= \lim_{\omega} \epsilon(f_n)\,\chi_{f_n}.
\]
\end{thm}

\smallskip

The following proposition enables us to recover the asymptotic of multipliers from the multiplier for the limit map. (cf. \cite{DMM08,MCM09,luo2021trees})

\smallskip

For each $n\in\N$, denote by $\Per_{l}(f_n)\subseteq \P^{1}(\C)$  the set of periodic points whose period divides $l$, repeated with multiplicities. We define an equivalence relation on $\prod_{n\in\N}\Per_{l}(f_n)$: two sequences are equivalent if they are equal on an $\omega$ set. 

The set of equivalent classes is denoted by $\mathrm{P}_l.$
\begin{prop}
\label{prop: contmultiplier}
The canonical map 
$
\mathrm{P}_l \longrightarrow \Per_{l}(f_{\omega}^{\epsilon(f)})$ sending $(p_n)_{n\in\N}$ to $p_{\omega}^{\epsilon(f)}$
is bijective.

Furthermore, we have the convergence of multiplier \[\lim_{\omega}|(f_n^{l})'(p_n)|^{\epsilon(f_n)}=|(f_{\omega}^{l})'(p)|.\]
\end{prop}
\begin{proof}
We first prove the surjectivity.
Without loss of generality, we suppose $l=1.$ We also suppose the fixed points of $f_n$ and $f_{\omega}^{\epsilon(f)}$ are all contained in the chart $[z\colon 1].$ Let $f_n=[P_n\colon Q_n]$ be a normalized representation so that the maximal absolute value of coefficients of $P$ and $Q$ is equal to $1$. Then the zeros of $P_n-zQ_n$ correspond to the fixed points of $f_n.$ 

  Suppose there exists $(p_{n})_{n\in\N}$ such that $P_n(p_n)-p_nQ(p_n)=0.$ Then $P_{\omega}^{\epsilon(f)}(p_{\omega}^{\epsilon(f)})-p_{\omega}^{\epsilon(f)}Q_{\omega}^{\epsilon(f)}(p_{\omega}^{\epsilon(f)})=0.$ Hence, $p_{\omega}^{\epsilon(f)}$ is a fixed point of $f_{\omega}^{\epsilon(f)}.$

  Conversely, let $p$ be a fixed point of $f_{\omega}^{\epsilon(f)}$, i.e., $P_{\omega}^{\epsilon(f)}(p)-pQ_{\omega}^{\epsilon(f)}(p)=0.$  By Proposition~\ref{prop: grouche}, there exists a sequence $(p_{n})_{n\in\N}\in\C^{\N}$  such that 
  \begin{enumerate}
      \item $p_{\omega}^{\epsilon(f)}=p$;
      \item $P_n(p_n)-p_nQ_n(p_n)=0$. 
  \end{enumerate}
  Note that $(f_n'(p_n))_{n\in\N}$ defines the equivalent class $(f_{\omega}^{\epsilon(f)})'(p_{\omega}^{\epsilon(f)})$. By the definition of the norm on the complex Robinson field $\sH^{\epsilon(f)}(\omega)$, we obtain that $\lim_\omega|f_n'(p_n)|^{\epsilon(f_n)}=|(f_{\omega}^{\epsilon(f)})'(p_{\omega}^{\epsilon(f)})|.$

  To prove bijectivity, it suffices to show that 
\(|\mathrm{P}_l| \le d^l + 1\),
which equals the cardinality of 
\(\Per_l(f_{\omega}^{\epsilon(f)})\).

Fix an ordering 
\(\Per_{l}(f_n) = \{p_{1,n}, \dots, p_{d^l+1,n}\}\).
Then any element 
\((p_n)_{n\in\N} \in \prod_{n\in\N} \Per_{l}(f_n)\)
can be written as 
\(p_n = p_{\phi(n),n}\)
for some function 
\(\phi \colon \N \to \{1, \dots, d^l+1\}\).
Define 
\(E_i = \{n \in \N \mid \phi(n) = i\}\);
the sets \(E_i\) form a finite covering of \(\N\),
so there exists \(i\) with \(E_i \in \omega\).
Hence 
\((p_n)_{n\in\N}\)
is equivalent to 
\((p_{i,n})_{n\in\N}\),
showing that each equivalence class is represented by one of these sequences.
\end{proof}

Denote by $[f_n]$ the conjugacy class of $f_n$ in the moduli space. Recall that  the resultant of the conjugacy class is defined as \[|\res([f_n])|\coloneqq \max_{M\in\PGL_{2}(\C)}|\Res(M_n^{-1}\circ f_n\circ M_n)|.\]
\begin{thm}
\label{thm: limitbadreduction}
If \( |\Res(f_n)| = |\res([f_n])| \) for all \(n\), 
then the rational map \( f_{\omega}^{\epsilon(f)} \) does not admit potential good reduction.
\end{thm}

\begin{proof}
See~\cite[Theorem 1.1]{FG25}.
\end{proof}
\subsection{Dichotomy in the asymptotic behavior of multipliers}
\label{sec: multiplier}
In this section, we prove Theorem~\ref{thm: multiplierdicho}.

\begin{proof}[Proof of Theorem~\ref{thm: multiplierdicho}]
Replacing $f_n$ with its conjugate, we may assume that $|\Res(f_n)| = |\res([f_n])| $. Since $(f_n)_{n\in\N}$ degenerates in the moduli space, we have $\lim_{\omega}\epsilon(f_n)=0.$  By Proposition~\ref{prop: complexrobinsonfield} and Theorem~\ref{thm: limitbadreduction}, $\sH^{\epsilon(f)}(\omega)$ is non-Archimedean, and \( f_{\omega}^{\epsilon(f)} \) does not have potential good reduction.

Let $\chi_{f_{\omega}^{\epsilon(f)}}$ be the Lyapunov exponent of $f_{\omega}^{\epsilon(f)}$. By Theorem~\ref{thm: asymtlyapunov}, we have 
\[\lim_{\omega}\epsilon(f_n)\chi_{f_n}=\chi_{f_\omega^{\epsilon(f)}}\]

By Theorem~\ref{thm: dichotomylyapunov}, if $\chi_{f_{\omega}^{\epsilon(f)}}>0$, then for any $\varepsilon>0$,  there exists $C>1$ such that for $l\in\N$ large enough, one has
 \[\#\left\{p \in \Per_{l}(f_{\omega}^{\epsilon(f)}) \middle| \lambda(f_{\omega}^{\epsilon(f)},p) >\frac{\log C}{l}+\frac{\chi_{f^{\epsilon(f)}_\omega}}{2} \right\} \geq (1 - \varepsilon) d^l.\]  
 Let $0<\alpha<\frac{1}{2}\chi_{f_{\omega}^{\epsilon(f)}}$. By Proposition~\ref{prop: contmultiplier}, for any Type-1 periodic point $p\in\Per_{l}(f_{\omega}^{\epsilon(f)})$ with $\lambda(f,p)>\frac{\log C}{l}+\frac{\chi_{f_{\omega}^{\epsilon(f)}}}{2}$, there exist $(p_n)_{n\in\N}\in\mathrm{P}_l$  such that $p=p_{\omega}^{\epsilon(f)}$, and that for $l$ large enough, we have 
\[\lim_{\omega}\epsilon(f_n)\lambda(f_n,p_n)=\lambda(f_{\omega}^{\epsilon(f)},p)> \frac{\log C}{l}+\frac{\chi_{f_\omega^{\epsilon(f)}}}{2}>\alpha.\]
Since $\lim_{\omega}\epsilon(f_n)=0$, we obtain $\lim_{\omega}\frac{-\log|\res([f_n])|}{\epsilon(f_n)}=1.$ Hence, one has \[\lim_{\omega}\frac{\lambda(f_n,p_n)}{-\log|\res([f_n])|}>\alpha.\]

Therefore, on an $\omega$-big set, $\lambda(f_n,p_n)>-\alpha\log|\res([f_n])|$. Thus, 
$$\#\Bigl\{ (p_n)\in \mathrm{P}_l : \lambda(f_n,p_n) > -\alpha\log|\res([f_n])|\ \text{ $\omega$-a.s.}\Bigr\}\geq  (1-\varepsilon)d^l.$$
We conclude that 
$$\lim_{l\to\infty}\frac{1}{d^l}\#\Bigl\{ (p_n)\in \mathrm{P}_l : \lambda(f_n,p_n) > -\alpha\log|\res([f_n])|\ \text{ $\omega$-a.s.}\Bigr\}=1.$$

Suppose $\chi_{f_\omega^{\epsilon(f)}}=0$. It follows from Theorem~\ref{thm: dichotomylyapunov} that \( f_{\omega}^{\epsilon(f)} \) is a Bernoulli map.  
Let \( I = [a,b] \) denote the convex hull of the Julia set 
\(\cJ(f_{\omega}^{\epsilon(f)})\), 
and let \( d_1, \dots, d_k \) be the integer slopes introduced in 
\S\ref{sec: Bernoullimap}.

For each \( l \in \mathbb{N} \), we claim that  \( f_{\omega}^{\epsilon(f)} \) admits at least
\[
d^l - d_1^l - d_k^l + 1
\]
Type-1 periodic points of period dividing \( l \), whose projections to \( I \) belong to the open interval \( (a,b) \).

Recall from \S\ref{sec: Bernoullimap} that 
\( f_{\omega}^{\epsilon(f)} \) maps the set of endpoints \( \{a,b\} \) to itself.  
We prove the claim under the assumption that 
\( f_{\omega}^{\epsilon(f)}(a) = a \) and \( f_{\omega}^{\epsilon(f)}(b) = b \); 
the remaining cases are analogous.

Since 
\(\deg_a \big((f_{\omega}^{\epsilon(f)})^l\big) = d_1^l\) 
and 
\(\deg_b \big((f_{\omega}^{\epsilon(f)})^l\big) = d_k^l\), 
Theorem~\ref{thm: positionfixedpoints} implies that there are 
\(d_1^l\) (resp. \(d_k^l\)) Type-1 fixed points of 
\( (f_{\omega}^{\epsilon(f)})^l \) 
whose projections equal \( a \) (resp. \( b \)).  
As the total number of Type-1 fixed points of 
\( (f_{\omega}^{\epsilon(f)})^l \) is \( d^l + 1 \), 
we obtain the desired estimate.

Furthermore, by Theorem~\ref{thm: positionfixedpoints}, 
if the projection of $p\in\Per_{l}(f_{\omega}^{\epsilon(f)})$ onto \( I \), denoted by \( x \), lies in the open interval \( (a,b) \), 
then there exists a root of unity 
\( \zeta\in \C \) 
such that
\[
\widetilde{((f_{\omega}^{\epsilon(f)})^l)'(p)} = \delta\,\tilde{(\zeta)},
\]
where \( \delta \) denotes the local degree of $(f_{\omega}^{\epsilon(f)})^l$ at \( x \).

By Proposition~\ref{prop: contmultiplier}, we can lift \( p \) to 
$(p_n)_{n\in\N} \in \mathrm{P}_l.$
By Lemma~\ref{lem: redlimit}, we have
\[
\bigl((f_{n}^l)'(p_n)\bigr)_{\omega}^{\C} = \delta\zeta.
\]

Since $\log \delta \le l \log \max\{d_1,\dots,d_k\},$
it follows that for any constant  satisfying 
\( C > \log \max\{d_1,\dots,d_k\} \), we have
\[
\lim_{\omega} \lambda(f_n, p_n) 
= \frac{1}{l} \log \delta 
< C.
\]
Hence,  we obtain
\[
\#\left\{ (p_n)_{n \in \N} \in \mathrm{P}_l 
\,\middle|\, 
\lambda(f_n, p_n) < C \text{ $\omega$-a.s.} 
\right\}
\ge d^l - d_1^l - d_k^l + 1.
\]

Therefore, for any integer
$\max\{d_1, d_k\} < m < d,$ we conclude 
 \[\#\left\{ (p_n)_{n \in \N} \in \mathrm{P}_l \,\middle|\, \lambda(f_n, p_n) < C \text{ $\omega$-a.s.} \right\} 
\geq d^l - m^l\] for $l$ large enough.
\end{proof}

\begin{rmk}
\label{rmk: optimalconstant}
 Suppose $f_{\omega}^{\epsilon(f)}$ is Bernoulli. Let $d_1,\cdots,d_k$ be integers as introduced in \S\ref{sec: Bernoullimap}. Then one can choose $C$ to be any constant which is greater than $\log\max_{1\leq i\leq d}\{d_i\}$. This choice is optimal, since \[\sup_{l}\frac{1}{l}\log\left\{\deg_{x}(f_{\omega}^{\epsilon(f)})^l,\text{ $x\in I$ and $(f_{\omega}^{\epsilon(f)}$})^l(x)=x\right\}=\log\max\{d_1,\cdots,d_k\}.\]
\end{rmk}

\subsection{Finiteness of multiplier scales}
 \label{sec: multiplierscale}
 In this section, we prove Theorem~\ref{thm: intro2d-2}. 

  Recall  a multiplier scale of a sequence of periodic points $(p_n)_{n\in\N}\in \mathrm{P}_l$ is defined by \[\eta_{p,n} = \frac{1}{1+\lambda(f_n, p_n)}.
\]

 Note that the scale $\eta_p$ is trivial if and only if $\lim_{\omega} \lambda(f_n, p_n)<\infty$. 

 \begin{lemma}
     \label{lem: scalecompare}
     Let $\eta_p$ be a multiplier scale defined by $(p_n)_{n\in\N}\in\mathrm{P}_l$, then we have $\eta_p\geq \epsilon(f).$
 \end{lemma}
\begin{proof}
 By Proposition~\ref{prop: contmultiplier}, we have 
\[\lim_{\omega}\epsilon(f_n)\lambda(f_n,p_n)=\lambda(f_\omega^{\epsilon(f)},p_{\omega}^{\epsilon(f)})\in\R_+.\]
Hence $\lim_{\omega}\epsilon(f_n)(1+\lambda(f_n,p_n))<\infty.$ Thus, we obtain $\eta_p\geq \epsilon(f).$
\end{proof}

\begin{lemma}
\label{lem: transcendentaldegree}
    Up to conjugation by a Möbius transformation, the rational map \( f_\omega^{\epsilon(f)} \) can be written in the form
    $f_\omega^{\epsilon(f)}(z) = \left[ \sum_{i=0}^{d} a_i z^i : \sum_{i=0}^{d} b_i z^i \right],$   where the field $\C(a_i,b_j) \subseteq\sH^{\epsilon(f)}(\omega)$
has transcendence degree at most $2d-2$ over $\C.$
\end{lemma}
\begin{proof}
Since \(\deg f_{\omega}^{\epsilon(f)} \geq 2\), the map \(f_{\omega}^{\epsilon(f)}\) admits at least three fixed points (counted with multiplicities), which we denote by \(p_1, p_2, p_3\). Write
\[
f_\omega^{\epsilon(f)}(z) = \left[ \sum_{i=0}^{d} a_i z^i : \sum_{i=0}^{d} b_i z^i \right].
\]
After normalization, we may assume that at least one of the coefficients is equal to \(1\).

We consider the possible configurations of fixed points:

  If \(p_1, p_2, p_3\) are distinct, we may suppose that  they are \(0\), \(1\), and \(\infty\). Under this normalization, we have \(a_0 = 0\), \(b_d = 0\), and the fixed point condition imposes the constraint \(\sum_{i=0}^{d} a_i = \sum_{i=0}^{d} b_i\). Therefore, the field \(\C(a_i, b_j)\) has transcendence degree at most \(2d - 2\).

  If \(p_1 = p_2 \ne p_3\), we may assume \(p_1 = p_2 = 0\) and \(p_3 = \infty\). Then the multiplicity of the fixed point at \(0\) implies \(a_0 = a_1 = 0\), and we also have \(b_d = 0\). Again, this yields the transcendence degree of $\C(a_i,b_j)$ is at most $2d-2.$

  If \(p_1 = p_2 = p_3\), we may assume all fixed points coincide at \(0\). Then the multiplicity condition implies \(a_0 = a_1 = a_2 = 0\). In this case, the transcendence degree of \(\C(a_i, b_j)\) remains at most \(2d - 2\).
\end{proof}
\begin{proof}[Proof of Theorem~\ref{thm: intro2d-2}]
Consider the algebraic closure $K$ of the subfield $\Q(a_i,b_j)$ in $\sH^{\epsilon(f)}(\omega)$ where $a_i,b_j$ are the coefficients
of the homogenous polynomials defining $f_{\omega}^{\epsilon(f)}$ in homogeneous coordinates.

By Lemma~\ref{lem: transcendentaldegree}, we may suppose that $K$ has transcendence degree at most $2d-2$ over $\C$.

Recall that for each non-trivial scale $\eta\geq \epsilon(f)$, one can define valuation rings $R_{+}^{\eta}\cap K,R_-^{\eta}\cap K\subseteq K$ , where \[
R_{+}^{\eta} = \left\{(z_n) \in \mathbb{C}^{\mathbb{N}} \mid \lim_{\omega} |z_n|^{\eta_n} <\infty\right\} \Big/ \left\{(z_n) \mid \lim_{\omega} |z_n|^{\epsilon(f_n)} = 0\right\}.
\]
and 
\[
R_{-}^{\eta} = \left\{(z_n) \in \mathbb{C}^{\mathbb{N}} \mid \lim_{\omega} |z_n|^{\eta_n} \leq 1 \right\} \Big/ \left\{(z_n) \mid \lim_{\omega} |z_n|^{\epsilon(f_n)} = 0\right\}.
\]

Let $\eta_p$ be a non-trivial multiplier scale defined by  $(p_n)_{n\in\N} \in \mathrm{P}_l$. By Lemma~\ref{lem: scalecompare}, we have $\eta_p\geq \epsilon(f).$ Note that 
\[\lim_{\omega}|(f_n^l)'(p_n)|^{\eta_{p,n}}=e>1.\] Hence, 
the sequence of multipliers $((f_n^l)'(p_n))_{n\in\N}$ belongs to $R_+^{\eta_p}\cap K \setminus R^{\eta_p}_-\cap K.$

We conclude using Theorem~\ref{thm: numbervaluationring}. 
\end{proof}

\section{Examples}
\label{sec: example}
In this section, we discuss in detail  several examples listed in the introduction.
\begin{example}
\normalfont
\label{eg: fundamental}
In the case (1) of Theorem~\ref{thm: multiplierdicho}, if we choose $\alpha>\frac{1}{2}\chi_{f_\omega^{\epsilon(f)}}$, then the statement might be false.

Consider \( f_n(z) = z^2 + n \).  
A direct computation of \( |\Res(f_n)| \) shows that 
$\epsilon(f) = \left(\frac{1}{\log n}\right)_{n \in \mathbb{N}}.$

Let \( a \) be the element of the complex Robinson field 
\(\sH^{\epsilon(f)}(\omega)\) represented by the sequence 
\((n)_{n \in \mathbb{N}}\). 
Then $f_{\omega}^{\epsilon(f)}(z) = z^2 + a,$ where
$|a| = \lim_{\omega} |n|^{\frac{1}{\log n}}=e.$

By induction, one verifies that for every integer \(k \ge 1\),
\[
|(f_{\omega}^{\epsilon(f)})^k(z)| =
\begin{cases}
e^{2^{k-1}}, & \text{if } |z| < \sqrt{e},\\[4pt]
e^{2^{k}},   & \text{if } |z| > \sqrt{e}.
\end{cases}
\]
Hence all points with modulus different from \(\sqrt{e}\) escape to infinity, 
and the Type-1 periodic points of \(f_{\omega}^{\epsilon(f)}\) is the circle 
\(|z| = \sqrt{e}\).  Hence,
\begin{align*}
\bigl|((f_{\omega}^{\epsilon(f)})^l)'(p)\bigr|
&= \prod_{i=0}^{l-1}
   \bigl|(f_{\omega}^{\epsilon(f)})'\bigl((f_{\omega}^{\epsilon(f)})^i(p)\bigr)\bigr| \\
&= 2^l e^{l/2}.
\end{align*}

By the Misiurewicz–Przytycki formula, the Lyapunov exponent of 
\(f_{\omega}^{\epsilon(f)}\) is given by 
\[
\chi_{f_{\omega}^{\epsilon(f)}} 
= \lim_{k \to \infty} 
   \frac{1}{2^k}\log\bigl|(f_{\omega}^{\epsilon(f)})^k(0)\bigr|
= \frac{1}{2}.
\]
Consequently, by Proposition~\ref{prop: contmultiplier}, for any 
\(l \in \mathbb{N}\) and  any
\((p_n)_{n \in \mathbb{N}} \in \mathrm{P}_l\), we have
\[
\lim_{\omega} \epsilon(f_n)\lambda(f_n,p_n)
= \frac{1}{2}
= \chi_{f_{\omega}^{\epsilon(f)}}.
\]
\end{example}
\begin{example}
\label{eg: 1}
\normalfont
Unlike in the case of a meromorphic family. It may occur that most multipliers are bounded, yet there exist exponentially many periodic points whose multipliers explode.

Consider the sequence of rational maps
\[
f_n(z)=\frac{z^2+n}{1+e^{-n}z^4}.
\]
One can verify by a computation of $|\Res(f_n)|$ that 
\(\epsilon(f)=(\frac{1}{n})_{n\in\mathbb{N}}\).
Let \(a,b\in\mathscr{H}^{\epsilon(f)}(\omega)\) be the elements defined respectively by the sequences \((e^{-n})_{n\in\mathbb{N}}\) and \((n)_{n\in\mathbb{N}}\). Then
\[
f_{\omega}^{\epsilon(f)}(z)=\frac{z^2+b}{1+az^4},
\]
with 
\(|a|=\lim_{\omega}|e^{-n}|^{\frac{1}{n}}=e^{-1}\) 
and \(|b|=\lim_{\omega}|n|^{\frac{1}{n}}=1\).

We claim that \(f_{\omega}^{\epsilon(f)}\) has no potential good reduction and is in fact a Bernoulli map:

For \(t\in\bigl[0,\tfrac{1}{2}\bigr]\), we have
\[
f_{\omega}^{\epsilon(f)}\bigl(\zeta(0,e^{t})\bigr)
=\zeta(0,e^{S(t)}),
\]
where
\[
S(t)=
\begin{cases}
    2t,&0\leq t\leq \tfrac{1}{4},\\[0.5em]
    -2t+1,&\tfrac{1}{4}\leq t\leq \tfrac{1}{2}.
\end{cases}
\]
In particular, \(x_g\) is a \emph{repelling fixed point of local degree \(2\)}, and the segment
\[
I=[x_g,\zeta(0,e^{\frac{1}{2}})]
\]
is backward invariant. Therefore, $f_{\omega}^{\epsilon(f)}$ is Bernoulli.

Hence, we are in the second case of Theorem~\ref{thm: multiplierdicho}, which implies that most multipliers remain bounded.

By Theorem~\ref{thm: positionfixedpoints}, for each integer \(l\geq 1\), there exist \(2^l\) Type-1 repelling periodic points in \(\Per_{l}\bigl(f_{\omega}^{\epsilon(f)}\bigr)\) whose projection onto \(I\) is \(x_g\). Note that
\[
\widetilde{f_{\omega}^{\epsilon(f)}}(z)=z^2+\tilde{b}.
\]

For each such periodic point \(p\), on the one hand, we can lift \(p\) to a sequence \((p_n)_{n\in\mathbb{N}}\in\prod_{n=0}^{\infty}\Per_{l}(f_n)\). On the other hand, since \(\tilde{p}\) is a periodic point of 
\[
g(z)=z^2+\tilde{b},
\]
Lemma~\ref{lem: liftfixedpt} and Proposition~\ref{prop: grouche} guarantee the existence of \((q_n)_{n\in\mathbb{N}}\in\prod_{n=0}^{\infty}\Per_l(g_n)\) such that
\(\widetilde{q_{\omega}^{\epsilon(f)}}=\tilde{p}\), where $g_n(z)=z^2+n.$  
Since \(p\) determines a good direction for \((f_{\omega}^{\epsilon(f)})^l\) at \(x_g\), we obtain
\begin{equation}\label{eq:multiplier}
\widetilde{\bigl((f_{\omega}^{\epsilon(f)})^l\bigr)'(p)}
=\bigl((\widetilde{f}_{\omega}^{\epsilon(f)})^l\bigr)'(\tilde{p})
=(g^l)'(\widetilde{q}_{\omega}^{\epsilon(f)}).
\end{equation}
Note that  if two sequences \((x_n)_{n\in\mathbb{N}}\) and \((y_n)_{n\in\mathbb{N}}\) in \(\mathbb{C}^{\mathbb{N}}\) satisfy 
\(\widetilde{x}_{\omega}^{\epsilon(f)}=\widetilde{y}_{\omega}^{\epsilon(f)}\), then
\[
\lim_{\omega}|x_n|^{\frac{1}{\log n}}
=\lim_{\omega}|y_n|^{\frac{1}{\log n}},
\]
because \(\lim_{\omega}\epsilon(f_n)\log n=0\).  
Thus, from \eqref{eq:multiplier} we deduce
\[
\lim_{\omega}\bigl|(f_{n}^l)'(p_n)\bigr|^{\frac{1}{\log n}}
=\lim_{\omega}\bigl|(g_n^l)'(q_n)\bigr|^{\frac{1}{\log n}}.
\]

By Example~\ref{eg: fundamental}, we have 
\[
\lim_{\omega}\bigl|(g_{n}^l)'(q_n)\bigr|^{\frac{1}{\log n}}=e^{\frac{l}{2}}
\]
Consequently, for $0<\alpha<\frac{1}{2}$, we have
\[
 \frac{1}{2^l} 
\#\Bigl\{ (p_n) \in \mathrm{P}_l : 
\lambda(f_n, p_n) > \alpha \log n 
\quad \omega\text{-a.s.} \Bigr\} 
\ge 1.
\]
\end{example}
\begin{example}
\label{eg: 2}
\normalfont
When \( d = 2 \), one may observe that up to three distinct multiplier scales can occur. 
Hence, the bound stated in Theorem~\ref{thm: intro2d-2} is optimal. 

To illustrate this, consider the sequence of rational maps
\[
f_n(z) = \frac{z(1-z)}{e^{-n}z^2 - n z + n}.
\]
It is straightforward to verify that the fixed points of \( f_n \) are given by
\begin{align*}
    p_{1,n} &= 0,\\
    p_{2,n} &= \frac{n - 1 - \sqrt{(n - 1)^2 - 4 e^{-n}(n - 1)}}{2 e^{-n}},\\
    p_{3,n} &= \frac{n - 1 + \sqrt{(n - 1)^2 - 4 e^{-n}(n - 1)}}{2 e^{-n}}.
\end{align*}
Next, we compute the multipliers at these fixed points. 
A direct differentiation yields
\begin{align*}
  f_n'(p_{1,n}) &= \frac{1}{n},\\[4pt]
  f_n'(p_{2,n}) &= n 
  - \frac{4 e^{-n}}
  {\left(\sqrt{1 - \dfrac{4 e^{-n}}{n - 1}} - 1\right)^{2}},\\
  f_n'(p_{3,n}) &= n 
  - \frac{4 e^{-n}}
  {\left(\sqrt{1 - \dfrac{4 e^{-n}}{n - 1}} + 1\right)^{2}}.
\end{align*}
Therefore, the multiplier scales associated with the sequences 
\((p_{1,n})_{n \in \mathbb{N}}\), \((p_{2,n})_{n \in \mathbb{N}}\), \((p_{3,n})_{n \in \mathbb{N}}\)
are respectively \((1)\), \((\frac{1}{n})_{n\in\N}\), and \((\frac{1}{\log n})_{n\in\N}\).
\end{example}
\bibliographystyle{alpha}
\bibliography{ref}

\begin{thebibliography}{QWY12}

\bibitem[Ben19]{Ben19}
Robert~L. Benedetto.
\newblock {\em Dynamics in one non-{Archimedean} variable}, volume 198 of {\em Grad. Stud. Math.}
\newblock Providence, RI: American Mathematical Society (AMS), 2019.

\bibitem[Ber90]{berkovich2012spectral}
Vladimir~G. Berkovich.
\newblock {\em Spectral theory and analytic geometry over non-{Archimedean} fields}, volume~33 of {\em Math. Surv. Monogr.}
\newblock Providence, RI: American Mathematical Society, 1990.

\bibitem[BH88]{BH88}
Bodil Branner and John~H. Hubbard.
\newblock The iteration of cubic polynomials. {I}: {The} global topology of parameter space.
\newblock {\em Acta Math.}, 160(3-4):143--206, 1988.

\bibitem[BR10]{BR10}
Matthew Baker and Robert Rumely.
\newblock {\em Potential theory and dynamics on the {Berkovich} projective line}, volume 159 of {\em Math. Surv. Monogr.}
\newblock Providence, RI: American Mathematical Society (AMS), 2010.

\bibitem[CN74]{comfort2012theory}
W.~Wistar Comfort and Stelios Negrepontis.
\newblock {\em The theory of ultrafilters}, volume 211 of {\em Grundlehren Math. Wiss.}
\newblock Springer, Cham, 1974.

\bibitem[DLU05]{DLU05}
Robert~L. Devaney, Daniel~M. Look, and David Uminsky.
\newblock The escape trichotomy for singularly perturbed rational maps.
\newblock {\em Indiana Univ. Math. J.}, 54(6):1621--1634, 2005.

\bibitem[DM08]{DMM08}
Laura~G. DeMarco and Curtis~T. McMullen.
\newblock Trees and the dynamics of polynomials.
\newblock {\em Ann. Sci. {\'E}c. Norm. Sup{\'e}r. (4)}, 41(3):337--383, 2008.

\bibitem[Fav25]{Fav25}
Charles Favre.
\newblock Blow-up of multipliers in meromorphic families of rational maps, 2025.

\bibitem[FG25]{FG25}
Charles Favre and Chen Gong.
\newblock Non-archimedean techniques and dynamical degenerations.
\newblock {\em Peking Mathematical Journal}, 2025.

\bibitem[FRL06]{FRL06}
Charles Favre and Juan Rivera-Letelier.
\newblock Quantitative uniform distribution of points of small height on the projective line.
\newblock {\em Math. Ann.}, 335(2):311--361, 2006.

\bibitem[FRL10]{FR10}
Charles Favre and Juan Rivera-Letelier.
\newblock Ergodic theory of rational maps over an ultrametric field.
\newblock {\em Proc. Lond. Math. Soc. (3)}, 100(1):116--154, 2010.

\bibitem[FRL25]{FRL25}
Charles Favre and Juan Rivera-Letelier.
\newblock Rigidit{\'e}, expansion et entropie en dynamique non-archim{\'e}dienne ({Rigidity}, expansion and entropy in non-{Archimedean} dynamics).
\newblock Preprint, {arXiv}:2504.20280 [math.{DS}] (2025), 2025.

\bibitem[Hug24]{Hug24}
Valentin Huguin.
\newblock Moduli spaces of polynomial maps and multipliers at small cycles.
\newblock Preprint, {arXiv}:2412.19335 [math.{DS}] (2024), 2024.

\bibitem[Jon15]{Jonsson}
Mattias Jonsson.
\newblock Dynamics on {Berkovich} spaces in low dimensions.
\newblock In {\em Berkovich spaces and applications. Based on a workshop, Santiago de Chile, Chile, January 2008 and a summer school, Paris, France, June 2010}, pages 205--366. Cham: Springer, 2015.

\bibitem[JX23a]{JX23}
Zhuchao Ji and Junyi Xie.
\newblock Homoclinic orbits, multiplier spectrum and rigidity theorems in complex dynamics.
\newblock {\em Forum Math. Pi}, 11:37, 2023.
\newblock Id/No e11.

\bibitem[JX23b]{JX23a}
Zhuchao Ji and Junyi Xie.
\newblock The multiplier spectrum morphism is generically injective.
\newblock Preprint, {arXiv}:2309.15382 [math.{DS}] (2023), 2023.

\bibitem[Luo21]{luo2021trees}
Yusheng Luo.
\newblock Limits of rational maps, {{\(\mathbb{R}\)}}-trees and barycentric extension.
\newblock {\em Adv. Math.}, 393:46, 2021.
\newblock Id/No 108075.

\bibitem[Luo22]{YLuo22}
Yusheng Luo.
\newblock Trees, length spectra for rational maps via barycentric extensions, and {Berkovich} spaces.
\newblock {\em Duke Math. J.}, 171(14):2943--3001, 2022.

\bibitem[McM87]{McM87}
Curt McMullen.
\newblock Families of rational maps and iterative root-finding algorithms.
\newblock {\em Ann. Math. (2)}, 125:467--493, 1987.

\bibitem[McM88]{McM88}
Curt McMullen.
\newblock Automorphisms of rational maps.
\newblock Holomorphic functions and moduli {I}, {Proc}. {Workshop}, {Berkeley}/{Calif}. 1986, {Publ}., {Math}. {Sci}. {Res}. {Inst}. 10, 31-60 (1988)., 1988.

\bibitem[McM09]{MCM09}
Curtis~T. McMullen.
\newblock Ribbon {{\(\mathbb{R}\)}}-trees and holomorphic dynamics on the unit disk.
\newblock {\em J. Topol.}, 2(1):23--76, 2009.

\bibitem[Mil93]{JM93}
John Milnor.
\newblock Geometry and dynamics of quadratic rational maps (with an appendix by {J}. {Milnor} and {Tan} {Lei}).
\newblock {\em Exp. Math.}, 2(1):37--83, 1993.

\bibitem[Mil06]{Mi06}
John Milnor.
\newblock {\em Dynamics in one complex variable}, volume 160 of {\em Ann. Math. Stud.}
\newblock Princeton, NJ: Princeton University Press, 3rd ed. edition, 2006.

\bibitem[Oku15]{Ok15}
Y{\^u}suke Okuyama.
\newblock Quantitative approximations of the {Lyapunov} exponent of a rational function over valued fields.
\newblock {\em Math. Z.}, 280(3-4):691--706, 2015.

\bibitem[QWY12]{QWY12}
Weiyuan Qiu, Xiaoguang Wang, and Yongcheng Yin.
\newblock Dynamics of mcmullen maps.
\newblock {\em Advances in Mathematics}, 229(4):2525--2577, 2012.

\bibitem[Roe03]{roe2003lectures}
John Roe.
\newblock {\em Lectures on coarse geometry}, volume~31 of {\em Univ. Lect. Ser.}
\newblock Providence, RI: American Mathematical Society (AMS), 2003.

\bibitem[Sil98]{SJ98}
Joseph~H. Silverman.
\newblock The space of rational maps on {{\(\mathbf P^1\)}}.
\newblock {\em Duke Math. J.}, 94(1):41--77, 1998.

\bibitem[Vaq00]{MV00}
Michel Vaqui{\'e}.
\newblock Valuations.
\newblock In {\em Resolution of singularities. A research textbook in tribute to Oscar Zariski. Based on the courses given at the working week in Obergurgl, Austria, September 7--14, 1997}, pages 539--590. Basel: Birkh{\"a}user, 2000.

\end{thebibliography}
\end{document}